\documentclass[12pt]{amsart}
\usepackage{amsmath,amssymb,latexsym, amsthm, amscd, mathrsfs, stmaryrd} 
\usepackage[all]{xy}

\usepackage[all]{xy}
\usepackage{color}
\usepackage{hyperref}

\theoremstyle{definition}
\newtheorem{example}[subsubsection]{Example}
\newtheorem{rem}[subsubsection]{Remark}
\theoremstyle{plain}
\newtheorem{prop}[subsubsection]{Proposition}
\newtheorem{thm}[subsubsection]{Theorem}
\newtheorem{lemma}[subsubsection]{Lemma}
\newtheorem{cor}[subsubsection]{Corollary}

\newtheorem{propo}{Proposition}

\newtheorem{thrm}{Theorem}

\newcommand{\mbf}{\mathbf}

\newcommand{\e}{\mbf e}
\newcommand{\f}{\mbf f}
\newcommand{\h}{\mbf h}
\newcommand{\D}{\mbf d}
\newcommand{\aaa}{\underline{\alpha}}
\newcommand{\bbb}{\underline{\beta}}
\newcommand{\ep}{\epsilon}

\title[]{Quasi-split symmetric pairs of $U(\mathfrak{gl}_N)$\\ and \\their Schur algebras}

\author[Y. Li and J. Zhu]{Yiqiang Li and Jieru Zhu}

\address{Department of Mathematics\\ University at Buffalo, SUNY  \\Buffalo, NY 14260}
\email{yiqiang@buffalo.edu (Li)}

\address{Department of Mathematics\\ University at Buffalo, SUNY  \\Buffalo, NY 14260}
    \email{jieruzhu@buffalo.edu (Zhu)}


\begin{document}

\maketitle

\begin{abstract}
We establish explicit isomorphisms of two seemingly-different algebras, and their Schur algebras,  arising from the centralizers of  two different type B Weyl group actions in 
Schur-like dualities. 
We provide a presentation of the geometric counterpart of  the above Schur algebras in~\cite{BKLW18} specialized at $q=1$. 
\end{abstract}

\section{Introduction}

\subsection{Overview}
A classical result of I. Schur states that the action of the symmetric group $S_d$ fully centralizes the natural action of the complex general linear algebra $\mathfrak{gl}_N$ on the tensor space $(\mathbb C^N)^{\otimes d}$. As a result, representations for $\mathfrak{gl}_N$, which are summands of $(\mathbb C^N)^{\otimes d}$, are in bijection with representations for $S_d$. The Schur algebra of type A is the centralizer algebra of $S_d$ on $(\mathbb{C}^N)^{\otimes d}$.

The generalization to type B has an interesting twist: the orthogonal group does not fully centralize the action of the type B Weyl group on the tensor space. On one hand, the orthogonal group is known to centralize the action of the Brauer algebra.
On the other hand, the description of the centralizer of the type B Weyl group action, and its quantization,  is nontrivial, as it bears no relation to the orthogonal group.
Two historical approaches have emerged to tackle this problem  by Green~\cite{Green97} and Shoji-Sakamoto~\cite{SS99}, see also~\cite{BW18b, HS04, MS16}. 
In  Green's approach \cite[Sec 2.3]{Green97},  the centralizer is given by a subgroup of $GL_N$, its Lie algebra being the fixed-point subalgebra $\mathfrak{gl}_N^{\theta}$ of $\mathfrak{gl}_N$ under a certain involution $\theta$, and its quantum analogue and double centralizer property is made explicitly in \cite{BW18b}; while in 
the Shoji-Sakamoto's approach the centralizer is given as a homomorphic
image of a two-block subalgebra $U$ of $U(\mathfrak{gl}_N)$.  
The pairs $(U(\mathfrak{gl}_N), U)$ and $(U(\mathfrak{gl}_N), U(\mathfrak{gl}_N^{\theta}))$ turn out to be (infinitesimal) quasi-split symmetric pairs of type A~\cite{SP87}. 
The purpose of this paper is to establish explicit isomorphisms between $U$ and $U(\mathfrak{gl}_N^{\theta})$, and consequently on their respective Schur algebras,
despite the ostensibly different actions from the type B Weyl group.

\subsection{Main results}

The  involution $\theta$ is induced by the diagram involution, say $\theta$ for an abuse of notation, on the Dynkin graph of $\mathfrak{gl}_N$
sending the Chevelley generator $E_i$ to $F_{\theta(i)}$ and $F_i$ to $E_{\theta(i)}$.
The fixed-point subalgebra $\mathfrak{gl}_N^{\theta}$ are generated by the simplest possible elements fixed by $\theta$: $\mbf e_i= E_i+F_{\theta(i)}, \mbf f_i=F_i+E_{\theta(i)}$
and $\mbf d_i = H_i-H_{\theta(i)}$. 
But the defining relations with respect to these generators are unknown until now, even though their quantum version in~\cite{Letzter97} has an explicit presentation under a quantized analogues of $\mbf e_i, \mbf f_i, \mbf d_i$. 
To this end, 
we adapt Letzter's proof  in \cite{Letzter97} for the quantum case to give a presentation of  $U(\mathfrak{gl}_N^{\theta})$, 
the universal enveloping algebra of $\mathfrak{gl}_N^{\theta}$ with respect to $\mbf e_i, \mbf f_i, \mbf d_i$.

\begin{propo}[Propositions~\ref{presentation},~\ref{presentation2}]\label{main1}
The algebra $U(\mathfrak{gl}_N^{\theta})$ is generated by $\mathbf{e}_1,\dots,\mathbf{e}_n$, $\mathbf{f}_1,\dots,\mathbf{f}_n$, $\mathbf{d}_1,\dots, \mathbf{d}_{n+1}$ subject to a list of relations (\ref{d})-(\ref{fen}), when $N=2n+1$.
For $N$ even, there is a similar presentation for $U(\mathfrak{gl}^{\theta}_N)$; see (R1)-(R4). 
\end{propo}

It has been somewhat known 
that the algebra $U(\mathfrak{gl}_N^{\theta})$ and the two-block Levi subalgebra $U$ are isomorphic to each other, but the isomorphism is implicit. 
We further establish an explicit isomorphism between $U(\mathfrak{gl}_N^{\theta})$ and  $U$ in Lemma~\ref{phi} and subsequent comments. Here $e_i$, $f_i$, $h_i$ are the Chevalley generators of the first block $\mathfrak{gl}_{n+1}$, and $e_{\overline{i}}$, $f_{\overline{i}}$, $h_{\overline{i}}$ are the Chevalley generators of the second block $\mathfrak{gl}_{n}$ in $U$.

\begin{thrm}[Theorem~\ref{basisoffixedpoint}]
There is an explicit isomorphism $\phi: U(\mathfrak{gl}_N^{\theta}) \to U(\mathfrak{gl}_{n+1}\oplus \mathfrak{gl}_n)$ between the two algebras, given via
\begin{align*}
&\e_i \mapsto e_i+e_{\overline{i}}  \hspace{.3 in}\f_i \mapsto f_i+f_{\overline{i}}  \hspace{.3 in} (i<n),   \\
&\e_n \mapsto 2e_n  \hspace{.5 in} \f_n \mapsto f_n  \hspace{.3 in}
\D_i \mapsto h_i + h_{\overline{i}} \hspace{.1 in} (1\leq i \leq n) \hspace{.3 in} 
\D_{n+1} \mapsto 2h_{n+1}.
\end{align*}
\end{thrm}

While it is not hard to show that $\phi$ is a homomorphism. It is nontrivial to show that it is an isomorphism. The difficulty lies in the fact that 
the obvious commutative subalgebra generated by $\mbf d_i$ is not the Cartan subalgebra of $U(\mathfrak{gl}^{\theta}_N)$ via an easy dimension count. 

To describe the preimage of $\phi$ we employ the notion of root vectors in relation to a root space decomposition of $\mathfrak{gl}_N^{\theta}$ with respect to $\mbf d_i$'s. In particular, let $\epsilon_i=\mathbf{d}_i^*$, and for an element $\alpha$ in the span of $\epsilon_i$, the $\alpha$-root space is all elements $x$ in $\mathfrak{gl}_N^{\theta}$ such that $[\mathbf{d}_i,x]=\alpha(\mathbf{d}_i)x$. In Proposition~\ref{stalksofweights}, we gave a description of the root space decomposition of $\mathfrak{gl}_N^{\theta}$:
\begin{propo}[Proposition~\ref{stalksofweights}]
The root spaces of $\mathfrak{gl}_N^{\theta}$ is either $1$ or $2$-dimensional if $\alpha\neq 0$. The $0$-weight space is $N$-dimensional, spanned by $\mathbf{d}_1,\dots,\mathbf{d}_{n+1}$, as well as explicitly and recursively defined elements $\mathbf{h}_1,\dots,\mathbf{h}_n$ in (\ref{preimageofh}).
\end{propo}

This gives a precise description of the Cartan subalgebra, by giving additional elements which commute with the generators $\mathbf{d}_i$, as a basis of the $0$-weight space. When the $\alpha$-root space is $2$-dimensional, let $X_{\alpha}$ and $X_{\alpha}'$ be two explicit elements which span the $\alpha$-root space, and use $X_{\alpha}$ if the associated root space is $1$-dimensional. The roots $\alpha$ in $\mathfrak{gl}_N^{\theta}$ can be identified with roots in $\mathfrak{gl}_{n+1}$ and $\mathfrak{gl}_n$ using successive imbedding of Lie algebras. For proper indices $i$ and $j$, let $Y_{\epsilon_i-\epsilon_j}=E_{ij}$ be the matrix unit in $\mathfrak{gl}_{n+1}$, and $Z_{\epsilon_i-\epsilon_j}=E_{ij}$ in $\mathfrak{gl}_{n}$, where the notation distinction signals the specific block in $\mathfrak{gl}_N$. Then in Theorem~\ref{rootvectors} we gave explicit description of the preimage of $\phi$ in terms of root vectors, here $\mathbf{h}_{n+1}=\frac{1}{2}\mathbf{d}_{n+1}$:

\begin{thrm}[Theorem~\ref{rootvectors}]
The following is true
\begin{align*}
\phi(X_{\alpha})=Y_{\alpha} \hspace{.3 in}
\phi(X'_{\alpha})=Y_{\alpha}+Z_{\alpha} \hspace{.3 in}
\phi(\mathbf{h}_i)=h_i \hspace{.3 in}(1\leq i\leq n+1).
\end{align*}
\end{thrm}

Therefore, the preimage of the Cartan generators of $U(\mathfrak{gl}_N^{\theta})$ can be explicit: let $\mathbf{h}_{\overline{i}}=\mathbf{d}_i-\mathbf{h}_i$, then $\phi(\mathbf{h}_{\overline{i}})=h_{\overline{i}}$. 
Theorem~\ref{basisoffixedpoint} follows from this observation.
To this end, in Proposition~\ref{intertwinesweyl} and Proposition~\ref{intertwinesenveloping} we have the following diagram
\[
\begin{CD}
U(\mathfrak{gl}_N^{\theta}) @>>> S^j  @. \curvearrowright  @.  V^{\otimes d} @. \curvearrowleft^G  @. S^B_d\\
@V\phi VV @V\phi VV  @.  @VV\cong V @. @VV= V \\
U @>>> S @. \curvearrowright  @.  V^{\otimes d} @. \curvearrowleft^{SS}  @. S^B_d\\
\end{CD}
\]
where $S^B_d$ is the type B Weyl group, and the first row is Green's approach while the second row is via Shoji-Sakamoto. 

Unpublished work of Kujawa-Zhu established a presentation for the centralizer, say $S $, in Shoji-Sakamoto's approach: this is a quotient of the algebra $U(\mathfrak{gl}_{n+1}\oplus\mathfrak{gl}_n)$ subject to further relations on the Cartan generators (\ref{hrelations1}) - (\ref{hrelations3}). 
Via the isomorphism $\phi$, we therefore write down their counterparts, say $S^j$, in Green's approach as relations in $S$ (see Proposition~\ref{newpresentation}) as follows.

\begin{propo}[Proposition~\ref{newpresentation}]\label{main-1}
The centralizer algebra $S^j$ in Green's approach is isomorphism to the quotient algebra of $U(\mathfrak{gl}_N^{\theta})$ by the ideal generated by the following relations.
\begin{align*}
\mathbf{d}_1+\cdots+\mathbf{d}_n+\mathbf{h}_{n+1}=&d,  \\
\h_i(\h_i-1)\cdots(\h_i-d)=&0 , \hspace{.5 in} 1\leq i\leq n+1,\\
\h_{\overline{i}}(\h_{\overline{i}}-1)\cdots(\h_{\overline{i}}-d)=&0,  \hspace{.5 in} 1\leq i\leq n. 
\end{align*}
\end{propo}

Note it is widely believed that the standard quantum algebra $U_q(\mathfrak{gl}_N^{\theta})$ and Letzter's nonstandard quantization $U^i_q$ (or $U^j_q$)~\cite{Letzter97} are not isomorphic to each other.
Surprisingly, a recent work of Lai-Nakano-Xiang \cite{LNX19} shows that the quantum deformations of the centralizers $S$ and $S^j$ in  both constructions are indeed isomorphic to each other.  
It is not  immediately clear how our approach is related to~\cite{LNX19}.

\subsection{Geometric implications}

Our main results are secretly motivated by a geometric question in~\cite{BKLW18} as we explained as follows.

The quantum version of the Schur algebra has a geometric construction by Beilinson-Lusztig-MacPherson \cite{BLM}. If $\mathcal{Y}$ is the variety of complete flags in a $d$-dimensional vector space over a finite field $\mathbb{F}_q$, and $\mathcal{X}$ is the variety of $N$-step flags in $\mathbb{F}_q^d$, then the Schur algebra can be realized geometrically: the Hecke algebra of type A, as a deformation of the group algebra of $S_d$, can be now realized as $GL_d$-invariant functions on the double flag variety $\mathcal{Y}\times \mathcal{Y}$. The $q$-Schur algebra of type A now consists of $GL_d$-invariant functions on the variety $\mathcal{X}\times \mathcal{X}$. The centralizing action of the $q$-Schur algebra and Hecke algebra on the $q$-tensor space, can be realized as a convolution between functions on $\mathcal{X}\times \mathcal{X}$, $\mathcal{X}\times \mathcal{Y}$ and $\mathcal{Y}\times \mathcal{Y}$.

Meanwhile, Bao-Kujawa-Li-Wang \cite{BKLW18} gave a geometric construction of a quantization of the type B q-Schur algebra $S^j_q$. In particular, the isotropic flags $\mathcal{X}_B$ and $\mathcal{Y}_B$ yield the type B Hecke algebra, similar to type A. 
In the quantum case, the quantum algebra behind the algebra of $O_N$-invariant functions on $\mathcal{X}_B\times \mathcal{X}_B$, denoted $U^j_q$ or $U^i_q$ depending on a parity, has a presentation using Chevalley generators, and can be imbedded inside the type A quantum group. Similar to type A, this algebra surjects onto the type B q-Schur algebra $S^j_q$. This algebra has been studied intensively in \cite{BW18,LW18}, and it admits a canonical basis.
It is an open problem to provide a presentation of $S^j_q$, see~\cite[Section~3.5]{BKLW18}.
It is know that $S^j$ is the specialization of $S^j_q$ at $q=1$.

\begin{propo}
Proposition~\ref{main-1} answers affirmatively the presentation problem in~\cite[Section~3.5]{BKLW18} when $q$ is specialized at $1$ and further disproves the naive expectation in {\it loc. cit.}
\end{propo}

In addition, the presentation in Proposition~\ref{main-1} provides insights into a potential full list of defining relations for $S^j_q$, which has yet to be explored; see~\cite{Li19} where the rank one case is settled.

\subsection*{Acknowledgment} Y. Li was partially supported by the NSF grant DMS-1801915. We thank the referee for their insightful comments towards several improvements of this article.

\tableofcontents

\section{The algebra $U^j$}
In this section, we aim to define the fixed point subalgebra of the general linear Lie algebra $\mathfrak{gl}_N$, where $N=2n+1$ is a positive odd integer. We will give a presentation of its universal enveloping algebra $U^j$ in nonstandard generators, and construct an explicit homomorphism from $U^j$ to the two-block Levi Lie subalgebra of $U(\mathfrak{gl}_N)$, the universal enveloping algebra of $\mathfrak{gl}_N$. This homomorphism turns out to be in fact an isomorphism, with an explicit description of the inverse image of the each Chevalley generator. In the end, we will describe the associated type B Schur algebra and give a presentation in terms of generators and relations.
\subsection{The fixed point subalgebra} \label{S2.1}
In this section we introduce the fixed point subalgebra of $\mathfrak{gl}_{2n+1}(\mathbb{C})$ with respect to an involution on its Dynkin diagram. We will further give a presentation of this fixed point subalgebra. Let $\Gamma$ be the Dynkin diagram  associated to the Lie algebra $\mathfrak{g}=\mathfrak{gl}_{2n+1}(\mathbb{C})$. It has $2n$ vertices $1,2,\dots,2n$ and a graph automorphism $\tau \in \operatorname{Aut}(\Gamma)$ such that $\tau(i)=2n+1-i$. Let $E_1,\dots,E_{2n},F_{1},\dots,F_{2n},H_1,\dots,H_{2n+1}$ be the Chevalley generators of $\mathfrak{g}$ and $\mathfrak{h}\subset \mathfrak{g}$ be the Cartan subalgebra spanned by $H_i (1\leq i\leq 2n+1)$. In other words, $E_i=E_{i,i+1}$, $F_i=F_{i+1,i}$, $H_i=E_{i,i}$, where $E_{i,j}$ is the matrix unit with a single nonzero entry $1$ in the $(i,j)$-position. The automorphism $\tau$ induces an automorphism of the Lie algebra $\mathfrak{gl}_{2n+1}(\mathbb{C})$, which can be extended to an automorphism $\theta$ of $\mathfrak{g}=\mathfrak{gl}_{2n+1}(\mathbb{C})$ via 
\begin{align*}
\theta(E_{i})=F_{\tau(i)}, \hspace{.2 in}\theta(F_{i})=E_{\tau(i)},\hspace{.2 in}\theta(H_{i})= H_{\tau(i)+1}.
\end{align*}
In fact, for any subset $X$ of the vertices in $\Gamma$, and any graph automorphism $\tau$ on $\Gamma$ which leaves $X$ invariant, Kolb \cite[Theorem 2.5]{Kolb14} defined an automorphism $\theta(X,\tau)$ of the associated Lie algebra. When $X=\emptyset$ and $\tau$ is taken as above, the automorphism $\theta(X,\tau)$ specializes to the automorphism mentioned above.

Let
\begin{align*}
\mathfrak{g}^{\theta}=\{x\in \mathfrak{g}\hspace{.1 in} | \hspace{.1 in}x=\theta(x)\}, \hspace{.2 in}
\mathfrak{h}^{\theta}= \{h\in \mathfrak{h} \hspace{.1 in} | \hspace{.1 in} h=\theta(h)\}.
\end{align*}
The result in \cite[Lemma 2.8]{Kolb14} specializes to the following
\begin{lemma}\cite[Lemma 2.8]{Kolb14}\label{Kolb}
The subalgebra $\mathfrak{g}^{\theta}$ is generated by the following elements in $\mathfrak{h}^{\theta}$ and 
\begin{align*}
\{F_i+\theta(F_i) \hspace{.1 in} | \hspace{.1 in}1\leq i\leq 2n\}. 
\end{align*}
\end{lemma}
We now give a description of generators for $\mathfrak{h}^{\theta}$.
\begin{lemma}\label{hfixed}
The subalgebra $\mathfrak{h}^{\theta}$ is spanned by $H_i+H_{2n+2-i}$ for $1\leq i\leq n+1$.
\end{lemma} 
\begin{proof}
Let $h=c_1H_1+\cdots +c_{2n+1}H_{2n+1}$ be fixed by $\theta$, then by the definition of $\theta$
\begin{align*}
c_1H_1+\cdots +c_{2n+1}H_{2n+1}=c_{2n+1}H_1+\cdots+ c_{1}H_{2n+1}
\end{align*}
and $c_i=c_{2n+1-i}$, hence $h$ is a linear combination of $H_i+H_{2n+2-i}$.
\end{proof}
\begin{rem} As we will show later in Theorem~\ref{basisoffixedpoint}, there is a larger algebra containing $\mathfrak{h}^{\theta}$ which is abelian. Therefore $\mathfrak{h}^{\theta}$  is not a Cartan subalgebra of $\mathfrak{g}^{\theta}$.
\end{rem} 
 
 Let $U^j=U(\mathfrak{g}^{\theta})$ be the universal enveloping algebra of $\mathfrak{g}^{\theta}$. By the two lemmas above, $U^j$ is generated by elements
\begin{align}
\e_i & = E_i+F_{2n+1-i} ,\hspace{.5 in}1\leq i\leq n, \label{gen1}\\
\f_i & = F_i+E_{2n+1-i}, \hspace{.5 in}1\leq i\leq n,\label{gen2}\\
\D_i & = H_i+H_{2n+2-i}, \hspace{.5 in}1\leq i\leq n, \label{gen3}\\
\D_{n+1} & = 2 H_{n+1}. \label{gen4}
\end{align} 
 
We aim to give a presentation of $U^j$. The following relations are known to be true for the quantum case in \cite[Lemma 2.2]{Letzter97}.
\begin{prop}\label{relations}\label{presentation}
The algebra $U^j$ is generated by $\e_1,\dots,\e_n$, $\f_1,\dots,\f_n$, $\D_1,\dots,\D_{n+1}$ subject to the following relations
\begin{align}
\D_i\D_j&=\D_j\D_i, \label{d}\\ 
[\D_a,\e_j]&=(-\delta_{a,j+1}-\delta_{2n+2-a,j+1}+\delta_{a,j})\e_j, \label{de}\\
[\D_a,\f_j]&=(-\delta_{a,j}+\delta_{a,j+1}+\delta_{2n+2-a,j+1})\f_j, \label{df}\\
\e_i\f_j-\f_j\e_i&=\delta_{ij}(\D_i-\D_{i+1}) \hspace{.5 in}i<n, \label{ef}\\
\e_i^2\e_j-2\e_i\e_j\e_i+\e_j\e_i^2&=0\hspace{.5 in} |i-j|=1, \label{eYB}\\ 
\f_i^2\f_j-2\f_i\f_j\f_i+\f_j\f_i^2&=0 \hspace{.5 in} |i-j|=1, \label{fYB}\\
\e_i\e_j&=\e_j\e_i\hspace{.5 in} |i-j|>1, \label{ee}\\
\f_i\f_j&=\f_j\e_i\hspace{.5 in} |i-j|>1, \label{ff}\\
\e_n^2\f_n-2\e_n\f_n\e_n+\f_n\e_n^2&=-4\e_n, \label{efn}\\
\f_n^2\e_n-2\f_n\e_n\f_n+\e_n\f_n^2&=-4\f_n. \label{fen}
\end{align}
\end{prop}
\begin{proof}
These relations hold by a straightforward calculation. 

Let $B$ be the algebra generated by ${\e_1},\dots,\e_n$, $\f_1,\dots,\f_n$, $\D_1,\dots,\D_{n+1}$ subject to Relations (\ref{d}), (\ref{de}) and (\ref{df}), then $B$ has a basis  $\{w\D^{\mathbf{s}}\}$ where $w$ is a word in $\e_1,\dots,\e_n$, $\f_1,\dots,\f_n$, and $\mathbf{s}=(s_1,\dots,s_{n+1})\in \mathbb{Z}_{\geq 0}^{n+1}$ with $\D^{\mathbf{s}}=\D_1^{s_1}\cdots \D_{n+1}^{s_{n+1}}$. Define a filtration on $B$ by declaring $\operatorname{deg} \e_i = \operatorname{deg} \f_i=1$ and  $\operatorname{deg} \D_i=0$, then any element in $B$ can be written as a sum of homogeneous parts, and declare its degree to be the highest among its homogeneous parts. On the other hand, the algebra $U(\mathfrak{gl}_{2n+1}(\mathbb{C}))$ has a PBW-basis and admits a filtration. The degree of each monomials is defined via $\operatorname{deg}E_i=1$, $\operatorname{deg}F_i=0$, $\operatorname{deg}H_i=0$. As a subalgebra of $U(\mathfrak{gl}_{2n+1}(\mathbb{C}))$, $U^j$ inherits a filtration where degrees on monomials are defined similarly.

Let $I$ be the ideal in $B$ generated by Relations (\ref{ef})--(\ref{fen}). The previous lemma implies that there is a well defined surjection $\phi: B\to U^j$ with $I$ lying in the kernel. Now we claim that the kernel is exactly $I$. Assume by contradiction that 
\begin{align*}
x=\displaystyle\sum_{w,\textbf{s}}a_{w,\textbf{s}}w\D^{\textbf{s}}
\end{align*}
 is in the kernel of $\phi$, $x\not\in I$, $a_{w,\textbf{s}}\in \mathbb{C}$ and $x$ is of smallest degree in $B$. For notation purposes rewrite $\f_i=\e_{2n+1-i}$ for $1\leq i\leq n$. By the definition of $\phi$, if $w=\e_{i_1}\e_{i_2}\dots \e_{i_k}$, then
\begin{align*}
\phi(a_{w,\textbf{s}}w\D^{{\textbf{s}}})=a_{\textbf{s}} E_{i_1}E_{i_2}\cdots E_{i_k}\phi(\D^{{\textbf{s}}}) + \text{ lower terms in }U^j.
\end{align*}
In the above case, let $E^w=E_{i_1}E_{i_2}\cdots E_{i_k}$ and $I_0$ be the index set of $(w,\mathbf{s})$ for the top degree terms in $x$, then by degree considerations, 
\begin{align*}
\sum_{\textbf{s}\in I_0}a_{w,\textbf{s}} E^{w}\phi(\D^{{\textbf{s}}}) =0 \in U(\mathfrak{gl}_{2n+1}(\mathbb{C})).
\end{align*}
Moreover, based on the PBW basis of $U(\mathfrak{gl}_{2n+1}(\mathbb{C}))$ and the fact that $\phi(\D^{\mathbf{s}_1})=\phi(\D^{\mathbf{s}_2})$ if and only if $\mathbf{s}_1=\mathbf{s}_2$, for each fixed $\mathbf{s}$, we have
\begin{align*}
\sum_{w:a_{w,\mathbf{s}}\neq 0}a_{w,\textbf{s}} E^{w} =0 \in U(\mathfrak{gl}_{2n+1}(\mathbb{C}))
\end{align*}
and the above element is in the ideal generated by the Serre relations for the positive part of $U(\mathfrak{gl}_{2n+1})$. Comparing these relations and Relations (\ref{eYB})--(\ref{fen}), one can use the corresponding relations to rewrite 
\begin{align*}
\displaystyle\sum_{w:a_{w,\mathbf{s}}\neq 0}a_{w,\textbf{s}} w
\end{align*}
 as an element of strictly less degree, resulting an element $x'$ in the kernel of $\phi$, $x'\not\in I$, whose degree is less than $x$, contradicting the choice of $x$.
\end{proof}

\begin{rem}
For more details, see the proof of Lemma~\ref{relations2} in the next section.
\end{rem}

\subsection{A homomorphism between $U^j$ and $U(\mathfrak{gl}_{n+1}\oplus\mathfrak{gl}_n)$}
\begin{lemma}\label{phi}
There is a well-defined algebra homomorphism $\psi : U^j \to U$ given by 
\begin{align}
\psi(\e_i) &= f_i+f_{\overline{i}}  \hspace{.3 in}\psi(\f_i)= e_i+e_{\overline{i}}  \hspace{.3 in} (i<n),  \label{hombegin} \\
\psi(\e_n) &= 2f_n  \hspace{.5 in}\psi(\f_n)=e_n,  \\ 
\psi(\D_i) &=  -h_i - h_{\overline{i}} \hspace{.1 in}  (1\leq i\leq n) \hspace{.3 in} 
\psi(\D_{n+1})= -2h_{n+1}.\label{homend}
\end{align}
\end{lemma}
\begin{proof}
The relations \ref{d}, \ref{eYB}, \ref{fYB}, \ref{ee}, \ref{ff} are trivial to check. For \ref{de}, it boils down to 
 \begin{align*}
[\D_a,\e_j]&=(-\delta_{a,j+1}+\delta_{a,j})\e_j
\end{align*}
unless $a=n+1$ and $j=n$, which will be checked separately ($2n+2-a=j+1$ is equivalent to $a+j=2n+1$, and since $j \leq n$ and $a\leq n+1$, this only holds when $j=n$ and $a=n+1$.) The calculation is as follows:
\begin{align*}
[\psi(\D_a),\psi(\e_j)]=&[-h_a-h_{\overline{a}}, f_j+f_{\overline{j}}]\\
=&(-\delta_{a,j+1}+\delta_{a,j})f_j+f_{\overline{j}} =(-\delta_{a,j+1}+\delta_{a,j}) (\psi(\e_j)).
\end{align*}
For the case $a=n+1$ and $j=n$
\begin{align*}
[\psi(\D_{n+1}),\psi(\e_n)]=[-2h_{n+1},2f_{n}]=-4f_{n} =-2\psi(\e_n)
\end{align*}
which is the value of $(-\delta_{a,j+1}-\delta_{2n+2-a,j+1}+\delta_{a,j})\psi(\e_j)$. \\

Relation \ref{df} can be checked similarly. For relation \ref{ef}, the relation holds when $i\neq j$. When $i=j<n$,
\begin{align*}
[\psi(\e_i),\psi(\f_i)]=&[f_i+f_{\overline{i}},e_i+e_{\overline{i}}]=[f_i,e_i]+[f_{\overline{i}},e_{\overline{i}}]\\
=&-h_i+h_{i+1}-h_{\overline{i}}+h_{\overline{i+1}}=\psi(\D_i-\D_{i+1}).
\end{align*}

Relations \ref{efn} is checked as follows (\ref{fen} can be checked similarly.)
\begin{align*}
\psi(\e_n^2\f_n-2\e_n\f_n\e_n+\f_n\e_n^2)=&\psi(\e_n)([\psi(\e_n),\psi(\f_n)])+[\psi(\f_n),\psi(\e_n)]\psi(\e_n)\\
=&4(f_n[f_n,e_n]+[e_n,f_n]f_n)=4(f_n(-h_n+h_{n+1})+(h_n-h_{n+1})f_n)\\
=&4([h_n,f_n]-[h_{n+1},f_n])=4(-f_n-f_n)=-8f_n=-4\psi(\e_n).
\end{align*}
Lemma is proved.
\end{proof}

\begin{rem}
As we will show later in Theorem~\ref{basisoffixedpoint}, $\psi$ is in fact an isomorphism.
\end{rem}


Let $\mathbf{i}$ be the involution on $U(\mathfrak{gl}_{n+1}(\mathbb{C}))\otimes U(\mathfrak{gl}_{n}(\mathbb{C}))$ induced by the following involution on each factor: $U(\mathfrak{gl}_{n+1}(\mathbb{C}))\simeq U(\mathfrak{gl}_{n+1}(\mathbb{C}))$ via $e_i \mapsto f_i$ ($1\leq i\leq n$), $f_i \mapsto e_i$ ($1\leq i\leq n$), $h_i \mapsto -h_i$ ($1\leq i\leq n+1$), and $U(\mathfrak{gl}_{n}(\mathbb{C}))\simeq U(\mathfrak{gl}_{n}(\mathbb{C}))$ similarly. Let $\phi=\mathbf{i}\circ \psi$, then $\phi$ is given by 
\begin{align}
\label{phi}
\begin{split}
\phi(\e_i) &= e_i+e_{\overline{i}}  \hspace{.3 in}\phi(\f_i)= f_i+f_{\overline{i}}  \hspace{.3 in} (i<n)   ,\\
\phi(\e_n) &= 2e_n  \hspace{.5 in} \phi(\f_n)= f_n  ,\\ 
\phi(\D_i) &=  h_i + h_{\overline{i}} \hspace{.1 in} (1\leq i \leq n) \hspace{.3 in} 
\phi(\D_{n+1}) = 2h_{n+1}.
\end{split}
\end{align}

\begin{rem}
Notice that the relations in $U^j$ are preserved once $\mathbf{d}_{n+1}$ is shifted by a constant. In other words, for any $c\in \mathbb{C}$, the map which fixes all other generators and $\mathbf{d}_{n+1}\mapsto \mathbf{d}_{n+1}+c$, defines an algebra automorphism on $U^j$. Therefore the map $\phi$ can be alternatively defined via a constant shift on $\mathbf{d}_{n+1}$.
\end{rem}

\begin{cor}\label{homo}
After restricting to $\mathfrak{g}$, the above map induces a well-defined Lie algebra homomorphism $\phi|_{\mathfrak{g^{\theta}}}:\mathfrak{g^{\theta}}\to \mathfrak{gl}_{n+1}(\mathbb{C})\oplus \mathfrak{gl}_n(\mathbb{C})$.
\end{cor}

\subsection{Root vectors in $U^j$}
Next we aim to give a weight space decomposition of $\mathfrak{g}^{\theta}$. First let $\mathbf{d}_{n+1}'=\frac{1}{2}\mathbf{d}_{n+1}$. Let $\epsilon_i=\D_i^*$ be the element in $(\mathfrak{h}^{\theta})^*$ dual to $\D_i$ for $1\leq i\leq n$, and $\epsilon_{n+1}=(\D_{n+1}')^*$. 

We shall define root vectors in $\mathfrak{g}^{\theta}$ similar to the well-known root vectors in type A. Specifically, recall $h_1,\dots,h_{n+1}$ is a basis of the Cartan subalgebra of $\mathfrak{gl}_{n+1}(\mathbb{C})$, and $h_{\overline{1}},\dots,h_{\overline{n}}$ is a basis of the Cartan subalgebra of $\mathfrak{gl}_{n}(\mathbb{C})$. When context is clear and index makes sense, we identify $\epsilon_i$ with $h_i^* $ or $h_{\overline{i}}^*$. Let 
\begin{align}
\Phi_{n+1}=&\{\epsilon_i-\epsilon_j \hspace{.1 in}|\hspace{.1 in} 1\leq i,j\leq n+1, i\neq j\} \label{rootsing} \\
 \Pi_{n+1}=&\{\epsilon_i-\epsilon_j \hspace{.1 in}|\hspace{.1 in} 1\leq i<j\leq n+1, j-i=1\}, \notag
\end{align}
be the set of roots and simple roots in $\mathfrak{gl}_{n+1}(\mathbb{C})$. For any $\alpha\in \Phi_{n+1}$, the root vectors  $Y_{\alpha}\in\mathfrak{gl}_{n+1}(\mathbb{C})$ are defined recursively as follows. Recall that $e_1,\dots,e_n$ are the standard upper-triangular Chevalley generators in $\mathfrak{gl}_{n+1}(\mathbb{C})$. On a simple root $\alpha_i=\epsilon_i-\epsilon_{i+1}$, set $Y_{\alpha_i}=e_i$ for $1\leq i\leq n$. In general, let $\alpha=\epsilon_i-\epsilon_{j}$ where $i<j$, define $Y_{\alpha}$ recursively via
\begin{align}
Y_{\alpha}=Y_{(\epsilon_i-\epsilon_{i+1}
)+(\epsilon_{i+1}-\epsilon_j)}=[Y_{\epsilon_i-\epsilon_{i+1}},Y_{\epsilon_{i+1}-\epsilon_j}]=[e_i,Y_{\epsilon_{i+1}-\epsilon_j}]. \label{rootvectoringln}
\end{align}
Similarly, for $\alpha_i=\epsilon_i-\epsilon_{i+1}$, set $Y_{-\alpha_i}=f_i$, where $f_1,\dots,f_n$ are the standard lower-triangular Chevalley generators in $\mathfrak{gl}_{n+1}(\mathbb{C})$, and define $Y_{-\alpha}$ for negative roots recursively as follows
\begin{align}
Y_{-\alpha}=Y_{-(\epsilon_i-\epsilon_{i+1})-(\epsilon_{i+1}-\epsilon_j)}=[Y_{-(\epsilon_i-\epsilon_{i+1})},Y_{-(\epsilon_{i+1}-\epsilon_j)}]=[f_i,Y_{-\epsilon_{i+1}+\epsilon_j}] .\label{negativerootvectorgln}
\end{align}
Similarly, let 
\begin{align}
\Phi_{n}=&\{\epsilon_i-\epsilon_j \hspace{.1 in}|\hspace{.1 in} 1\leq i,j\leq n, i\neq j\} \hspace{.2 in} \label{rootsing2} \\
\Pi_{n}=&\{\epsilon_i-\epsilon_j \hspace{.1 in}|\hspace{.1 in} 1\leq i<j\leq n, j-i=1\} \notag
\end{align}
be the set of roots and simple roots in $\mathfrak{gl}_{n}(\mathbb{C})$. For any $\alpha\in \Phi_{n}$, the root vectors $Z_{\alpha}\in \mathfrak{gl}_{n}(\mathbb{C})$ are defined similarly.

It is well known that $\{Y_{\alpha}\}_{\alpha\in \Phi_{n+1}}\cup \{h_i\}_{i=1}^{n+1}$ form a basis for $\mathfrak{gl}_{n+1}(\mathbb{C})$, and $[h_i,Y_{\alpha}]=\alpha(h_i)Y_{\alpha}$. Similar results also hold for $Z_{\alpha}$.

In the same fashion, we construct root vectors in $\mathfrak{g}^{\theta}$. Recall that $\{H_1,\dots,H_{2n+1}\}$ is a basis of the Cartan subalgebra in $\mathfrak{gl}_{2n+1}(\mathbb{C})$. Let $\mu_i=H_i^*$ be the element in $\mathfrak{h}^*$ dual to  $H_i$ for $1\leq i \leq 2n+1$, and let
\begin{align}
\Phi_{2n+1}^+=&\{\mu_i-\mu_j \hspace{.1 in}| \hspace{.1 in} 1\leq i<j\leq 2n+1\}.\label{twosetsofroots}
\end{align}
We call $\Phi_{2n+1}^+$ the set of positive roots in $\mathfrak{gl}_{2n+1}$. For every positive root $\alpha\in \Phi^{+}_{2n+1}$, we construct $X_{\alpha}\in \mathfrak{g}^{\theta}$ recursively as follows.  Recall that $\e_i,\f_i,\D_i$ are elements in $\mathfrak{g}^{\theta}$ defined in $(\ref{gen1})-(\ref{gen4})$. On a simple root, let 
\begin{align*}
X_{\mu_i-\mu_{i+1}}=&\e_i \hspace{.2 in} (1\leq i\leq n)\\
X_{\mu_i-\mu_{i+1}}=&\f_{2n+1-i}\hspace{.2 in} (n+1\leq i\leq 2n).
\end{align*}
In general, for $\alpha=\mu_i-\mu_j$, define $X_{\alpha}$ similar to (\ref{rootvectoringln}):
\begin{align}
X_{\mu_i-\mu_j}=X_{(\mu_i-\mu_{i+1})-(\mu_{i+1}-\mu_j)}=[X_{\mu_i-\mu_{i+1}},X_{\mu_{i+1}-\mu_j}] .\label{rootvectorsinfixed}
\end{align}

Even though the definition uses a fixed order of splitting a positive root into a sum of simple roots, we argue that the result is independent of the order of the sum and its associated order of taking Lie brackets. Specifically, it is helpful to give an alternative formulation of the root vectors:
\begin{cor}\label{independentofbrackets}
The following holds for any  $1\leq i <j  \leq 2n+1$: if $\alpha=\mu_i-\mu_j$, then
\begin{align}
X_{\alpha}=X_{(\mu_{i}-\mu_{j-1})+(\mu_{j-1}-\mu_j)}=[X_{\mu_{i}-\mu_{j-1}},X_{\mu_{j-1}-\mu_j}]. \label{reversingbrackets}
\end{align}
\end{cor}

\begin{proof}
First, for $x,y,z\in \mathfrak{g}$, if $[x,z]=0$, then
\begin{align*}
[x,[y,z]]=-[y,[z,x]]-[z,[x,y]]=[[x,y],z].
\end{align*}
Fix $j$ and induct on $i$. The base cases $i=j-1$ and $i=j-2$ follow from definition. Suppose the statement is true for $k\leq i<j$, then
\begin{align*}
[X_{\mu_{k-1}-\mu_j}]=[X_{\mu_{k-1}-\mu_k},X_{\mu_k-\mu_j}]=[X_{\mu_{k-1}-\mu_k},[X_{\mu_k-\mu_{j-1}},X_{\mu_{j-1}-\mu_j}]]\\
=[[X_{\mu_{k-1}-\mu_k},X_{\mu_k-\mu_{j-1}}],X_{\mu_{j-1}-\mu_j}]]=[X_{\mu_{k-1}-\mu_{j-1}},X_{\mu_{j-1}-\mu_j}].
\end{align*}
The third equality follows from the fact that for any two simple roots not adjascent to each other in the sequence $\e_1,\dots,\e_n,\f_n,\dots,\f_1$, their associated root vectors commute with each other.
\end{proof}

For the set of simple roots
\begin{align*}
\Pi_{2n+1} = \{ \alpha_i=\mu_i-\mu_{i+1} \hspace{.1 in}| \hspace{.1 in} 1\leq i\leq 2n \}.
\end{align*}
We impose a total order based on the ordering on $i$. Moreover, two roots are said to be adjacent if they are adjacent on the associated Dynkin diagram of type $A_{2n}$. Each simple root vector induces a Lie algebra endomorphism in $\mathfrak{g}^{\theta}$ by its adjoint action: 
\begin{align*}
\operatorname{ad} X_{\alpha_i}: \mathfrak{g}^{\theta} \to \mathfrak{g}^{\theta} \hspace{.3 in} \operatorname{ad} X_{\alpha_i} (z)=[X_{\alpha_i},z].
\end{align*}
For short we denote $\operatorname{ad} X_{\alpha_i}$ simply as $\underline{\alpha}_i$. The following identity will be helpful in the future.

The sequence of simple roots $(\alpha_{i_1},\alpha_{i_2},\alpha_{i_3},\dots,\alpha_{i_s})$ is called \emph{admissible} if for any $k$ such that $1\leq k\leq s-1$, $\alpha_k$ is adjacent to one of $\alpha_{i_{k+1}}$, $\alpha_{i_{k+2}}$, \dots, $\alpha_{i_s}$.

\begin{lemma}\label{admissible}
Let $\alpha_{i_1}$,\dots, $\alpha_{i_s}$ be a sequence of simple roots. Then 
\begin{align*}
\aaa_{i_1}\circ \aaa_{i_1}\circ \cdots \circ \aaa_{i_s} (\mathbf{d}_k)\neq 0
\end{align*}
for some $k$, $1\leq k\leq n+1$, only if the sequence $(\alpha_{i_1},\dots,\alpha_{i_s})$ is admissible.
\end{lemma}
\begin{proof}
We prove by induction on $s$. Suppose the statement is true for $s\leq m$. For $s=m+1$, suppose there exists $\mathbf{d}_k$ such that
\begin{align*}
\aaa_{i_1}\circ\aaa_{i_2}\circ\cdots \circ \aaa_{i_{m+1}}(\mathbf{d}_k)\neq 0
\end{align*}
then the sequence $(\alpha_{i_2},\dots,\alpha_{i_{m+1}})$ is admissible by induction hypothesis. Suppose further that $\alpha_{i_1}$ is not adjacent to any of $\alpha_{i_2},\dots,\alpha_{i_{m+1}}$, then
\begin{align*}
\aaa_{i_1}\circ\aaa_{i_2}\circ\cdots \circ \aaa_{i_{m+1}}(\mathbf{d}_k)=[[X_{\alpha_{i_1}},X_{\alpha_{i_2}}],\aaa_{i_3}\circ\aaa_{i_2}\circ\cdots \circ \aaa_{i_{m+1}}(\mathbf{d}_k)]\\
+[\aaa_{i_1}\circ\aaa_{i_3}\circ\cdots \circ \aaa_{i_{m+1}}(\mathbf{d}_k),X_{\alpha_{i_2}}].
\end{align*}
In the sum, the first term is zero because $\alpha_{i_1}$ is nonadjacent to $\alpha_{i_2}$, and the second term is zero by induction hypothesis for $s=m$. Now we have shown that the claim is true for $s=m+1$.
\end{proof}

\begin{lemma}\label{shuffle}
 Let $(\beta_{1},\dots,\beta_{s})$ be a sequence of consecutive simple roots in increasing order. If $(\alpha_{i_1},\dots,\alpha_{i_s})$ is a permutation of the sequence $(\beta_{1},\dots,\beta_{s})$, then
\begin{align*}
\aaa_{i_1}\circ\aaa_{i_2} \circ \cdots\circ \aaa_{i_s}(\mathbf{d}_k)\in \mathbb{C} \bbb_{1}\circ \bbb_{2}\circ\cdots\circ \bbb_{s}(\mathbf{d}_k) 
\end{align*}
for all $1\leq k\leq n+1$.
\end{lemma}
\begin{proof}
First observe that if $\aaa_{i_1}\circ \aaa_{i_2}\circ \cdots\circ \aaa_{i_{t+1}}(\mathbf{d}_k)=0$, then the statement is automatically true. We now assume that this quantity is nonzero. It follows that the sequence $(\alpha_{i_2},\alpha_{i_{k+1}},\dots,\alpha_{i_s})$ must be admissible, and therefore must be a permutation of consecutive roots by Lemma~\ref{admissible}. 

We proceed by induction on $s$: the base case when $s=2$ is a straightforward check. Now suppose the statement holds for $s\leq t$. We claim it also holds for $s=t+1$. Let $\beta_1,\dots,\beta_t$ be the rearrangement of $\alpha_{i_2},\dots,\alpha_{i_{k+1}}$ in increasing order, then
\begin{align*}
x=\aaa_{i_1}\circ \aaa_{i_2}\circ \cdots\circ \aaa_{i_{t+1}}(\mathbf{d}_k)\in \mathbb{C} \aaa_{i_1}\circ \bbb_{1}\circ \cdots\circ \bbb_{t}(\mathbf{d}_k).
\end{align*}
Since $x\neq 0$, $\alpha_{i_1}$ is adjacent to one of $\beta_1,\dots,\beta_t$. By assumption there are no repeated roots, therefore $\alpha_{i_1}$ must be immediately smaller than $\beta_1$ or immediately larger than $\beta_t$. The former case yields the desired conclusion automatically. If the latter is true, then
\begin{align*}
&\aaa_{i_1}\circ \bbb_{1}\circ \cdots\circ \bbb_{t}(\mathbf{d}_k)\\
=&[[X_{\alpha_1},X_{\beta_1}],\bbb_2\circ\cdots\circ \bbb_t(\mathbf{d}_k)+[\aaa_{i_1}\circ\bbb_2 \circ \cdots\circ \bbb_t (\mathbf{d}_k),X_{\beta_1}].
\end{align*}
The first term is zero, and the second term is in $\mathbb{C}[X_{\beta_1},\bbb_2\circ \cdots\circ  \bbb_t \circ \aaa_{i_1}(\mathbf{d}_k)]$ by the induction hypothesis as desired.
\end{proof}

\begin{lemma}\label{keyrewrite}
For any $a,b,c\in \mathfrak{g}$
\begin{align*}
[a,[b,[a,c]]]=\frac{1}{2}[b,[a,[a,c]]]+\frac{1}{2}[c,[a,[a,b]]]+\frac{1}{2}[a,[a,[b,c]]].
\end{align*} 
\end{lemma}
\begin{proof}
This is a straightforward check.
\end{proof}

\begin{thm}\label{span}
The root vectors $\{X_{\alpha}\}_{\alpha\in\Phi_{2n+1}^+}$ and $\{\mathbf{d}_i\}_{1\leq i\leq n+1}$ span $\mathfrak{g}^{\theta}$.
\end{thm}
\begin{proof}
Let $(\mathfrak{g}^{\theta})'=[\mathfrak{g}^{\theta},\mathfrak{g}^{\theta}]$. Recall $\mathfrak{h}^{\theta}$ is the subalgebra spanned by $\mathbf{d}_i$, $1\leq i\leq n+1$. Since $[\mathfrak{g}^{\theta},\mathfrak{h}^{\theta}]\subset (\mathfrak{g}^{\theta})'$, it suffices to show root vectors span $(\mathfrak{g}^{\theta})'$.

Given an arbitrary element 
\begin{align}\label{generalelement}
\aaa_{i_1}\circ\aaa_{i_2}\circ \cdots \circ \aaa_{i_s}(\mathbf{d}_k) \in (\mathfrak{g}^{\theta})',
\end{align}
where $\alpha_{i_1}$, \dots, $\alpha_{i_s}$ are simple roots in $\mathfrak{gl}_{2n+1}(\mathbb{C})$ (with possible repetitions.) We claim it is in the span of root vectors whose expression only involves the roots $\alpha_{i_1}$, \dots, $\alpha_{i_s}$. 

We prove this by induction on $s$. Suppose the statement is true for $s=t$. Then for $s=t+1$, by Lemma~\ref{admissible}, if (\ref{generalelement}) is nonzero, then the sequence $(\alpha_{i_2},\dots,\alpha_{i_{t+1}})$ is admissible. If there is no repetition of simple roots among them, then the admissible condition implies that they are a permutation of a sequence of consecutive simple roots, and by Lemma~\ref{shuffle} is a root vector.

Now we discuss the case when there are repeated roots. List all roots which sit at the left endpoint of such a pair of repeated roots, and choose the rightmost one, $\alpha_{i_\ell}$, among them. By design, this root has a twin counterpart to its right, but no pairs of repeated roots to its right. By the induction on length, 
\begin{align}\label{general2}
\aaa_{i_{\ell+1}}\circ\cdots\circ \aaa_{i_{t+1}}(\mathbf{d}_k)
\end{align}
is in the span of root vectors concerning $\alpha_{i_{\ell+1}}$, \dots, $\alpha_{i_{t+1}}$ for any $k$. Therefore (\ref{generalelement}) is a linear combination of elements the form
\begin{align}
\aaa_{i_1}\circ\cdots\circ \aaa_{i_\ell} \circ \beta_1 \circ \beta_2\circ \cdots\circ \beta_p(\mathbf{d}_k) \label{general2}
\end{align}
where $(\beta_1,\beta_2,\cdots, \beta_p)$ is an increasing sequence of consecutive simple roots, and $\alpha=\alpha_{i_\ell}=\beta_j$ for some $1\leq j\leq p$. Notice that if two simple roots $\gamma_1$ and $\gamma_2$ are nonadjacent, then for any $z\in \mathfrak{g}^{\theta}$,
\begin{align*}
[X_{\gamma_1},[X_{\gamma_2},z]]=[[X_{\gamma_1},X_{\gamma_2}],z]+[[X_{\gamma_1},z],X_{\gamma_2}]=-[X_{\gamma_2},[X_{\gamma_1},z]],
\end{align*}
therefore up to a sign, one can move $\alpha_{i_{\ell}}$ past all nonadjacent roots to its right in (\ref{general2}), such that its rightmost end is of the form

1) $\aaa\circ \aaa \circ \underline{\gamma}(y)$,
where $\gamma$ is immediately larger than $\alpha$ and y is a root vector whose expression contains only roots nonadjacent to $\alpha$; or

2) $\aaa\circ \bbb \circ \aaa \circ \underline{\gamma}(y)$

where $(\beta,\alpha,\gamma)$ is an increasing sequence of consecutive simple roots, $y$ is a root vector whose expression only contains roots nonadjacent to $\alpha$ or $\beta$.

In Case 1),
\begin{align*}
[X_{\alpha},[X_{\alpha},[X_{\gamma},y]]]=[X_{\alpha},[[X_{\alpha},X_{\gamma}],y]]=[[X_{\alpha},[X_{\alpha},X_{\gamma}]],y]=c[X_{\alpha},y] .
\end{align*}
Here
\begin{align*}
c=\begin{cases}
-4  \hspace{.2 in} \text{if } \{\alpha,\beta\}=\{\ep_n-\ep_{n+1},-\ep_n+\ep_{n+1}\}\\
0 \hspace{.2 in}\text{otherwise}.
\end{cases}
\end{align*} 
Therefore (\ref{general2}) can be written as an element in the form of (\ref{generalelement}) of strictly lower length by the relations in $\mathfrak{g}^{\theta}$.

In Case 2)
\begin{align*}
\begin{split}
[X_{\alpha},[X_{\beta},[X_{\alpha},[X_{\gamma},y]]]]& =[X_{\alpha},[X_{\beta},[[X_{\alpha},X_{\gamma}],y]]]\\
&=[X_{\alpha},[[X_{\beta},[X_{\alpha},X_{\gamma}]],y]]=[[X_{\alpha},[X_{\beta},[X_{\alpha},X_{\gamma}]]],y].
\end{split}
\end{align*}
By Lemma~\ref{keyrewrite}
\begin{align}
[X_{\alpha},[X_{\beta},[X_{\alpha},X_{\gamma}]]]=\frac{1}{2}[X_{\beta},[X_{\alpha},[X_{\alpha},X_{\gamma}]]] ,\notag \\
+\frac{1}{2}[X_{\gamma},[X_{\alpha},[X_{\alpha},X_{\beta}]]]+\frac{1}{2}[X_{\alpha},[X_{\alpha},[X_{\beta},X_{\gamma}]]] .\label{general3}
\end{align}
The third term is of the form in Case 1), therefore (\ref{general3}) is in
\begin{align*}
\mathbb{C}[X_{\alpha},X_{\beta}]\oplus \mathbb{C}[X_{\alpha},X_{\gamma}].
\end{align*}
After taking the bracket with $y$, one can use the jacobi identity to rewrite the ending part in (\ref{general2}), so that it is of the form (\ref{generalelement}) of smaller length.
\end{proof}

For $\alpha\in\Phi_{n+1}$, an element $x\in \mathfrak{g}^{\theta}$ is a weight vector of weight $\alpha$, under the adjoint action of $\mathbf{d}_1,\dots,\mathbf{d}_n$, $\mathbf{d}_{n+1}'$, if and only if 
\begin{align*}
[\mathbf{d}_i,x]=\alpha(\mathbf{d}_i)x \hspace{.2 in} (1\leq i\leq n) \hspace{.3 in}
[\mathbf{d}'_{n+1},x]=\alpha(\mathbf{d}'_{n+1})x.
\end{align*}

Define a map $s: \Phi_{2n+1}^+\to \Phi_{n+1}\cup\{0\}$, such that on a simple root $\alpha \in \Pi_{2n+1}$, $X_{\alpha}$ is a weight vector of weight $s(\alpha)$. That is to say,
\begin{align*}
s(\mu_i-\mu_{i+1})=&\epsilon_i-\epsilon_{i+1} \hspace{.2 in} (1\leq i\leq n), \\
s(\mu_i-\mu_{i+1})=&-\epsilon_{2n+1-i}+\epsilon_{2n+2-i}  \hspace{.2 in} (n+1\leq i\leq 2n).
\end{align*} 
Extend $s$ linearly so that on a sum of simple roots $\alpha=\alpha_1+\cdots \alpha_i \in \Phi_{2n+1}^+$, 
\begin{align*}
s(\alpha)=s(\alpha_1)+\cdots+s(\alpha_i).
\end{align*}
The map $s$ specifies the weight of each root vector.

\begin{lemma}\label{correctweight}
Let $\alpha\in \Phi_{2n+1}^+$, then $X_{\alpha}$ is of weight $s(\alpha)$ under the adjoint action of $\mathbf{d}_1,\dots,\mathbf{d}_n,\mathbf{d}'_{n+1}$.
\end{lemma}
\begin{proof}
The statement is clear on a simple root vector. On a sum of simple roots, it follows from the calculation
\begin{align*}
[\mathbf{d}_i,[X_{\alpha},X_{\beta}]]=[[\mathbf{d}_i,X_{\alpha}],X_{\beta}]+[[\mathbf{d}_i,X_{\beta}],X_{\alpha}]=s(\alpha+\beta)(\mathbf{d}_i)[X_{\alpha},X_{\beta}].
\end{align*}
The lemma is proved.
\end{proof}

\begin{rem}
We can thus view $\Phi_{n+1}$ as the set of roots for $\mathfrak{g}^{\theta}$.
\end{rem}

Recall $\phi$ is the homomorphism defined in Corollary~\ref{homo}. We can describe the stalks of $s$ explicitly.
\begin{prop}\label{stalksofweights}
For the map $s$ defined above and $1\leq i<j\leq n$,
\begin{align}
s^{-1}(0)=&\{\mu_i-\mu_{2n+2-i}\hspace{.1 in}|\hspace{.1 in} 1\leq i\leq n\}, \notag \\
s^{-1}(\epsilon_i-\epsilon_{n+1})=&\{\mu_i-\mu_{n+1}\}, \hspace{.3 in}  \notag \\
s^{-1}(-\epsilon_i+\epsilon_{n+1})=&\{\mu_{n+1}-\mu_{2n+2-i}\}, \hspace{.3 in}  \notag \\
s^{-1}(\epsilon_i-\epsilon_j)=& \{\mu_i-\mu_j,\mu_i-\mu_{2n+2-j}\}, \notag\\
s^{-1}(-\epsilon_i+\epsilon_j)=& \{\mu_{2n+2-j}-\mu_{2n+2-i},\mu_{j}-\mu_{2n+2-i}\}.\notag
\end{align}
\end{prop}
\begin{proof}
Notice that
\begin{align*}
0=&(\epsilon_i-\epsilon_{i+1})+\cdots+(\epsilon_n-\epsilon_{n+1})+(-\epsilon_{n}+\epsilon_{n+1})+\cdots+(-\epsilon_i+\epsilon_{i+1}).
\end{align*}
Therefore the element simplifies to be 
\begin{align*}
&s^{-1}(\epsilon_i-\epsilon_{i+1})+\cdots+s^{-1}(\epsilon_n-\epsilon_{n+1})+s^{-1}(-\epsilon_{n}+\epsilon_{n+1})+\cdots+s^{-1}(-\epsilon_i+\epsilon_{i+1})\\
=&(\mu_i-\mu_{i+1})+\cdots+(\mu_n-\mu_{n+1})+(\mu_{n+1}-\mu_{n+2})+\cdots+(\mu_{2n+1-i}-\mu_{2n+2-i})\\
=&\mu_i-\mu_{2n+2-i},
\end{align*}
which is in $s^{-1}(0)$. Also, when $i<j<n+1$,
\begin{align*}
\epsilon_i-\epsilon_j=&(\epsilon_i-\epsilon_{i+1})+\cdots+(\ep_{j-1}-\ep_{j})\\
&+(\ep_j-\ep_{j+1})+\cdots+(\epsilon_n-\epsilon_{n+1})+(-\epsilon_n+\epsilon_{n+1})+\cdots+(-\epsilon_j+\epsilon_{j+1})
\end{align*}
Therefore, we have
\begin{align*}
&s^{-1}(\epsilon_i-\epsilon_{i+1})+\cdots+s^{-1}(\epsilon_n-\epsilon_{n+1})+s^{-1}(-\epsilon_n+\epsilon_{n+1})+\cdots+s^{-1}(-\epsilon_j+\epsilon_{j+1})\\
=&(\mu_i-\mu_{i+1})+\cdots+(\mu_n-\mu_{n+1})+(\mu_{n+1}-\mu_{n+2})+\cdots+(\mu_{2n+1-j}-\mu_{2n+2-j})\\
=&\mu_i-\mu_{2n+2-j},
\end{align*}
which is also in $s^{-1}(\epsilon_i-\epsilon_j)$, in addition to $\mu_i-\mu_j$.
The other cases can be checked via a similar calculation.
\end{proof}

\subsection{An algebra isomorphism with $U(\mathfrak{gl}_{n+1}\oplus\mathfrak{gl}_n)$}

Based on the previous lemma, we re-index the root vectors using weights of $\D_i$. In particular, for $1\leq i<j\leq n$, let 
\begin{align*}
X'_{\epsilon_i-\epsilon_j}=&X_{\mu_i-\mu_j} \hspace{.3 in}  X_{\epsilon_i-\epsilon_j}=\frac{1}{2}
X_{\mu_i-\mu_{2n+2-j}}\\
X'_{-\epsilon_i+\epsilon_j}=&X_{\mu_{2n+2-j}-\mu_{2n+2-i}} \hspace{.3 in}    X_{-\epsilon_i+\epsilon_j}=\frac{1}{2}X_{\mu_{j}-\mu_{2n+2-i}} .
\end{align*}
In addition, for $1\leq i \leq n$, let 
\begin{align*}
X_{\epsilon_i-\epsilon_{n+1}}=\frac{1}{2}X_{\mu_i-\mu_{n+1}}, \hspace{.3 in} X_{-\epsilon_i+\epsilon_{n+1}}=X_{-\mu_i+\mu_{n+1}}.
\end{align*}
Also, define $\mathbf{h}_i$ for $1\leq i\leq n+1$ recursively via  $\mathbf{h}_{n+1}=\mathbf{d}_{n+1}'$ and 
\begin{align}
 \mathbf{h}_i=\mathbf{h}_{i+1}+\frac{1}{2}X_{\mu_i-\mu_{2n+2-i}} \hspace{.1 in } (1\leq i\leq n). \label{preimageofh}
\end{align}

Recall $\Phi_{n+1}$ and $\Phi_n$ are the set defined in (\ref{rootsing}) and (\ref{rootsing2}), and $\phi$  is the map defined after  (\ref{phi}). The following lemma establishes a connection between the various root vectors.
\begin{lemma}\label{rootvectors}
The following holds
\begin{align}
\phi(X_{\alpha})=&Y_{\alpha} \hspace{.3 in}(\forall \alpha\in \Phi_{n+1} ), \label{imageofroot1}\\
\phi(X'_{\alpha})=&Y_{\alpha}+Z_{\alpha} \hspace{.3 in}(\forall \alpha\in \Phi_n),\label{imageofroot2}\\
\phi(\mathbf{h}_i)=&h_i \hspace{.3 in}(1\leq i\leq n+1). \label{imageofroot3}
\end{align}

\end{lemma}
\begin{proof}
We first prove $(\ref{imageofroot2})$ for a root $\alpha=\epsilon_i-\epsilon_j$ with $i<j$ by fixing $j$ and induction on $i$. The base case $i=j-1$ is true by definition:
\begin{align*}
\phi(X'_{\epsilon_{j-1}-\epsilon_j})=\phi(X_{\mu_{j-1}-\mu_j})=\phi(\e_{j-1})=e_{j-1}+e_{\overline{j-1}}=Y_{\epsilon_{j-1}-\epsilon_j}+Z_{\epsilon_{j-1}-\epsilon_j}.
\end{align*}
Suppose the statement is true for $i=k$, then
\begin{align*}
\phi(X_{\epsilon_{k-1}-\epsilon_j})=&\phi(X_{\mu_{k-1}-\mu_j})=\phi([X_{\mu_{k-1}-\mu_k},X_{\mu_k-\mu_j}]) \\
=&[e_{k-1}+e_{\overline{k-1}},Y_{\epsilon_k-\epsilon_j}+Z_{\epsilon_k-\epsilon_j} ]=Y_{\epsilon_{k-1}-\epsilon_j}+Z_{\epsilon_{k-1}-\epsilon_j}.
\end{align*}
The arguments  are similar for a negative root in (\ref{imageofroot2}).

To show $(\ref{imageofroot1})$ for a root $\alpha=\epsilon_i-\epsilon_{n+1}$ for $1\leq i\leq n$, we induct on $i$. The base case when $i=n$ is as follows
\begin{align*}
\phi(X_{\epsilon_n-\epsilon_{n+1}})=\phi(\frac{1}{2}X_{\mu_n-\mu_{n+1}})=\frac{1}{2}\phi(\e_n)=e_n=Y_{\epsilon_n-\epsilon_{n+1}}.
\end{align*}
Suppose the statement is true for $i=k$,
\begin{align*}
\phi(X_{\epsilon_{k-1}-\epsilon_{n+1}})=&\phi(\frac{1}{2}X_{\mu_{k-1}-\mu_{n+1}})=\frac{1}{2}\phi([X_{\mu_{k-1}-\mu_k},X_{\mu_k-\mu_{n+1}}])=\phi([\e_{k-1},X'_{\epsilon_k-\epsilon_{n+1}}])\\
=&[e_{k-1}+e_{\overline{k-1}},Y_{\epsilon_k-\epsilon_{n+1}}]=Y_{\epsilon_{k-1}-\epsilon_{n+1}}.
\end{align*}
The arguments are similar for a negative root in (\ref{imageofroot2}).

We now show that (\ref{imageofroot3}) is true by induction on $i$. The case when $i=n+1$ is given via definition. The next case when $i=n$ is verified as follows:
\begin{align*}
\phi({\mathbf{h}_n})=&\phi(\mathbf{h}_{n+1}+\frac{1}{2}X_{\mu_n-\mu_{n+2}})=\phi(\mathbf{h}_{n+1})+\frac{1}{2}[\phi(X_{\mu_n-\mu_{n+1}}),\phi(X_{\mu_{n+1}-\mu_{n+2}})]\\
=&\phi(\mathbf{h}_{n+1})+\frac{1}{2}[\phi(X_{\epsilon_n-\epsilon_{n+1}}),\phi(X_{-\epsilon_n+\epsilon_{n+1}})]=h_{n+1}+[e_n,f_n]=h_n
\end{align*}

Using $i=n$ as the base case, suppose the statement is true for $i$, then by Lemma~\ref{independentofbrackets},
\begin{align*}
\phi(\mathbf{h}_{i-1})=&\phi(\mathbf{h}_{i})+\phi(\frac{1}{2}X_{\mu_{i-1}-\mu_{2n+3-i}})\\
=&\phi(\mathbf{h}_i)+\frac{1}{2}\phi([X_{\mu_{i-1}-\mu_i},[X_{\mu_i-\mu_{2n+2-i}},X_{\mu_{2n+2-i}-\mu_{2n+3-i}}]])\\
=&h_i+\phi([\mathbf{e}_{i-1},[\mathbf{h}_i-\mathbf{h}_{i+1},X_{-\epsilon_{i-1}+\epsilon_i}]])=h_i+\phi([\mathbf{e}_{i-1},[\mathbf{h}_{i}-\mathbf{h}_{i+1},\mathbf{f}_{i-1}]])\\
=&h_i+[Y_{\epsilon_{i-1}-\epsilon_i}+Z_{\epsilon_{i-1}-\epsilon_i},[h_i-h_{i+1},Y_{\epsilon_{i-1}+\epsilon_i}+Z_{\epsilon_{i-1}+\epsilon_i}]]\\
=&h_i+[Y_{\epsilon_{i-1}-\epsilon_i}+Z_{\epsilon_{i-1}-\epsilon_i},Y_{-\epsilon_{i-1}+\epsilon_i}]=h_i+h_{i-1}-h_i=h_{i-1},
\end{align*}
and therefore the statement is also true for $i-1$. Notice the root $\mu_i-\mu_{2n+2-i}=0$ if and only if $i=n+1$, and the induction starts at $i=n$.

We now show (\ref{imageofroot1}) via fixing $j$ and induction on $i$. The base case for $(i,j)$ when $i=j-1$ (or $j=i+1$) is as follows
\begin{align*}
\phi(X'_{\epsilon_i-\epsilon_{i+1}})=&\phi(\frac{1}{2}X_{\mu_i-\mu_{2n+2-(i+1)}})=\frac{1}{2}\phi([X_{\mu_i-\mu_{i+1}},X_{\mu_{i+1}-\mu_{2n+2-(i+1)}}])\\
=&\frac{1}{2}\phi([\mathbf{e}_i,2(\mathbf{h}_{i+1}-\mathbf{h}_{i+2})])=[Y_{\epsilon_i-\epsilon_{i+1}}+Z_{\epsilon_i-\epsilon_{i+1}},h_{i+1}-h_{i+2}]=Y_{\epsilon_i-\epsilon_{i+1}}.
\end{align*}
Now suppose (\ref{imageofroot1}) holds for the index pair $(i,j)$, then we claim it also holds for $(i-1,j)$:
\begin{align*}
\phi(X_{\epsilon_{i-1}-\epsilon_j})=&\phi(\frac{1}{2}X_{\mu_{i-1}-\mu_{2n+2-j}})=\frac{1}{2}\phi([X_{\mu_{i-1}-\mu_i},X_{\mu_i-\mu_{2n+2-j}}])\\
=&\phi([\mathbf{e}_{i-1},X'_{\epsilon_i-\epsilon_j}])=[Y_{\epsilon_{i-1}-\epsilon_i}+Z_{\epsilon_{i-1}-\epsilon_i},Y_{\epsilon_i-\epsilon_j}]=Y_{\epsilon_{i-1}-\epsilon_j}.
\end{align*}
Therefore we have shown (\ref{imageofroot1}) by induction. The negative roots follow a similar argument.
\end{proof}

Recall the definition of $\mathbf{h}_i$ in (\ref{preimageofh}).

\begin{thm}\label{basisoffixedpoint}
The morphism $\phi$ (resp. $\phi|_{\mathfrak{g}^{\theta}}$)  in (\ref{phi}) is an isomorphism of (resp. Lie) algebras. Moreover, 
the following set is a basis of $\mathfrak{g}^{\theta}$. 
\begin{align*}
\{X_{\alpha}\}_{\alpha\in \Phi_{n+1}}\cup \{X'_{\alpha}\}_{\alpha\in \Phi_n}\cup\{\mathbf{h}_1,\dots,\mathbf{h}_{n+1}\}\cup \{\D_1,\dots,\D_n\},
\end{align*}
 and the elements $\D_1$, \dots, $\D_{n}$, $\h_1$, \dots, $\h_{n+1}$ pairwise commute. 
\end{thm}

\begin{proof}
These elements are exactly those mentioned in Theorem~\ref{span}, hence they span $\mathfrak{g}^{\theta}$. By Lemma~\ref{rootvectors}, they are also linearly independent, because their image forms a linearly independent set.
\end{proof}

\subsection{The type B Schur algebra}

Recall the elements $\mathbf{h}_i$, defined in terms of generators of $U^i$ in (\ref{preimageofh}). For $1\leq i\leq n$, let
$
\h_{\overline{i}}=\mathbf{d}_{\overline{i}}-\h_{i} $. By the definition of $\phi$ and Lemma~\ref{rootvectors}, $\phi(\h_{\overline{i}})=h_{\overline{i}}$. Now fix a positive integer $d$. Let $S$ be a quotient of $U(\mathfrak{gl}_{n+1}(\mathbb{C})\oplus\mathfrak{gl}_n(\mathbb{C}))$, by the following relations
\begin{align}
h_1+\cdots+h_{n+1}+h_{\overline{1}}+\cdots+h_{\overline{n}}&=d,\label{hrelations1}\\
h_{i}(h_i-1)\cdots(h_{i}-d)&=0, \hspace{.5 in} 1\leq i\leq n+1,\label{hrelations2}\\
h_{\overline{i}}(h_{\overline{i}}-1)\cdots(h_{\overline{i}}-d)&=0, \hspace{.5 in} 1\leq i\leq n. \label{hrelations3}
\end{align}
These relations are analogous to the type A relations, i.e. relations in a presentation for the type A Schur algebras given by Doty-Giaquinto \cite{DG02}.

On the other hand, let $S^j$ be the quotient of $U^j$ under further relations
\begin{align}
\mathbf{d}_1+\cdots+\mathbf{d}_n+\mathbf{h}_{n+1}=&d, \label{sumofd} \\
\h_i(\h_i-1)\cdots(\h_i-d)=&0, \hspace{.5 in} 1\leq i\leq n+1,\\
\h_{\overline{i}}(\h_{\overline{i}}-1)\cdots(\h_{\overline{i}}-d)=&0, \hspace{.5 in} 1\leq i\leq n. \label{hnobar}
\end{align}
Note: once unraveling the notation of root vectors, these relations are purely in terms of generators $\mathbf{e}_i$, $\mathbf{f}_i$ and $\mathbf{d}_i$ of $U^j$.
\begin{prop}\label{newpresentation}
The algebras $S^j$ and $S$ are isomorphic.
\end{prop}

\begin{rem}
{The algebra $S^j$ is isomorphic to the centralizer of the type B Weyl group studied by Green~\cite{Green97}.}
\end{rem}

\begin{rem}
By an unpublished work of Kujawa-Zhu, the hyperoctahedral Schur algebra admits a presentation as a quotient using relations (\ref{hrelations1})-(\ref{hrelations3}). Therefore, the above proposition provides an alternative presentation of the Schur algebra of type B.
\end{rem}

\begin{example}
When $n=1$, $S^j$ is the algebra with generators $\mathbf{e}_1$, $\mathbf{f}_1$, $\mathbf{d}_1$, $\mathbf{d}_2$ subject to the relations in Proposition~\ref{presentation} and the following additional relations.
\begin{align*}
\mathbf{d}_1+\frac{1}{2}\mathbf{d}_2=&d,\\
\mathbf{d}_2(\mathbf{d}_2-2)\cdots (\mathbf{d}_2-2d)=&0,\\
(\mathbf{d}_2+[\mathbf{e}_1,\mathbf{f}_1])(\mathbf{d}_2+[\mathbf{e}_1,\mathbf{f}_1]-2)\cdots(\mathbf{d}_2+[\mathbf{e}_1,\mathbf{f}_1]-2d)=&0,\\
(\mathbf{d}_2+\frac{1}{2}[\mathbf{e}_1,\mathbf{f}_1]-d)(\mathbf{d}_2+\frac{1}{2}[\mathbf{e}_1,\mathbf{f}_1]
-d-1)\cdots (\mathbf{d}_2+\frac{1}{2}[\mathbf{e}_1,\mathbf{f}_1]
-2d)=&0.
\end{align*} 
\end{example}

\section{The algegbra $U^i$ } \label{oddsection}
In the case when $N$ is even, all major results in the previous section have their analogous counterparts in this section: a root space decomposition of the fixed point subalgebra, an isomorphism with the two-block Levi Lie subalgebra of $U(\mathfrak{gl}_N)$, an explicit description of inverse image of such an isomorphism using the notion of root vectors, and an alternative presentation of the type B Schur algebra.  

\subsection{The fixed point subalgebra}
We start by introducing the fixed point subalgebra and its presentation using generators similar to those in Section~\ref{S2.1}. We now discuss the case when the underlying Dynkin diagram has an odd number of vertices. We will abuse some of the previous notation despite this new assumption on the parity. In particular, let $\mathfrak{g}=\mathfrak{gl}_{2n}(\mathbb{C})$ and let $\Gamma$ be the Dynkin diagram of $\mathfrak{gl}_{2n}(\mathbb{C})$. Then $\Gamma$ has vertices $1,2,\dots,2n-1$. Let $\tau\in \operatorname{Aut}(\Gamma)$ be the graph automorphism such that $\tau(i)=2n-i$, then $\tau$ induces a Lie algebra automorphism on $\mathfrak{sl}_{2n}(\mathbb{C})$, which extends to a Lie algebra automorphism $\theta$ on  $\mathfrak{gl}_{2n}(\mathbb{C})$ via
\begin{align*}
\theta(E_i)=F_{\tau(i)},\hspace{.2 in} \theta(F_i)=E_{\tau(i)}, \hspace{.2 in} \theta(H_i)=H_{\tau(i)+1}.
\end{align*}
Let $\mathfrak{g}^{\theta}=\mathfrak{gl}_{2n}(\mathbb{C})^{\theta}$ be the subalgebra of $\mathfrak{g}$ fixed by $\theta$, and $E_i, F_i (1\leq i\leq 2n-1)$, $H_i (1\leq i\leq 2n)$ be the Chevalley generators of $U(\mathfrak{gl}_{2n})$. Define the following elements in $U(\mathfrak{gl}_{2n})$:
\begin{align*}
\mathbf{e}_i&=E_i+F_{2n-i} \hspace{.3 in} (1\leq i\leq n-1),\\ 
\mathbf{f}_i&=F_i+E_{2n-i} \hspace{.3 in} (1\leq i\leq n-1),\\
 \mathbf{t}&=E_n+F_n,\\
 \mathbf{d}_i&=H_i+H_{2n+1-i} \hspace{.3 in} (1\leq i\leq n).
\end{align*} 
Let $\mathfrak{h}\subset \mathfrak{g}$ be the Cartan subalgebra of $\mathfrak{gl}_{2n}(\mathbb{C})$, and $\mathfrak{h}^{\theta}=\{h\in \mathfrak{h}\hspace{.1 in}|\hspace{.1 in} \theta(h)=h\}$ the subalgebra of $\mathfrak{h}$ fixed by $\theta$. Similar to Lemma~\ref{hfixed}, we have the following
\begin{lemma}
The algebra $\mathfrak{h}^{\theta}$ is spanned by $\mathbf{d}_i$ $(1\leq i\leq n)$. 
\end{lemma}

By Lemma~\ref{Kolb}, the elements $\mathbf{e}_i$, $\mathbf{f}_i$, $\mathbf{t}$, $\mathbf{d}_i$ generate $\mathfrak{g}^{\theta}$. Let $U(\mathfrak{g}^{\theta})$ be the universal enveloping algebra of $\mathfrak{g}^{\theta}$.

\begin{prop}\label{relations2}\label{presentation2}
The algebra $U(\mathfrak{g}^{\theta})$ is generated by $\mathbf{e}_1,\dots,\mathbf{e}_{n-1}$, $\mathbf{f}_1,\dots,\mathbf{f}_{n-1}$, $\mathbf{t}$, $\mathbf{d}_1,\dots,\mathbf{d}_n$, subject to the following relations
\begin{align*}
\mathbf{t}\mathbf{d}_i&=\mathbf{d}_i \mathbf{t} \hspace{.5 in}(1\leq i\leq n),\tag{R1}\\
\mathbf{d}_i\mathbf{e}_j-\mathbf{e}_j\mathbf{d}_i&=(\delta_{i,j}-\delta_{i,j+1})\mathbf{e}_j,\\
\mathbf{d}_i\mathbf{f}_j-\mathbf{f}_j\mathbf{d}_i&=(-\delta_{i,j}+\delta_{i,j+1})\mathbf{f}_j,\\
\mathbf{t}\mathbf{e}_i & =\mathbf{e}_i \mathbf{t},\tag{R2}\\
\mathbf{t} \mathbf{f}_i & =\mathbf{f}_i \mathbf{t}\hspace{.5 in}  (1\leq i \leq n-2) ,\\
\mathbf{e}_i\mathbf{e}_j&=\mathbf{e}_j\mathbf{e}_i \hspace{.5 in} |i-j|>1,\\
\mathbf{f}_i\mathbf{f}_j&=\mathbf{f}_j\mathbf{f}_i \hspace{.5 in} |i-j|>1,\\
\mathbf{e}_i\mathbf{f}_j-\mathbf{f}_j\mathbf{e}_i&=\delta_{i,j}\mathbf{d}_i-\mathbf{d}_{i+1},\\
\mathbf{e}_{n-1}^2\mathbf{t}-2\mathbf{e}_{n-1}\mathbf{t}\mathbf{e}_{n-1}+\mathbf{t}\mathbf{e}_{n-1}^2 &=0 ,\tag{R3}\\
\mathbf{f}_{n-1}^2\mathbf{t}-2\mathbf{f}_{n-1}\mathbf{t}\mathbf{f}_{n-1}+\mathbf{t}\mathbf{f}_{n-1}^2 &=0,\\
\mathbf{e}_{n-1}\mathbf{t}^2-2\mathbf{t}\mathbf{e}_{n-1}\mathbf{t}+\mathbf{t}^2\mathbf{e}_{n-1}&=\mathbf{e}_{n-1}, \tag{R4}\\
\mathbf{f}_{n-1}\mathbf{t}^2-2\mathbf{t}\mathbf{f}_{n-1}\mathbf{t}+\mathbf{t}^2\mathbf{f}_{n-1}&=\mathbf{f}_{n-1}.
\end{align*} 
\end{prop}
\begin{proof}
R1): $[\mathbf{t},\mathbf{d}_i]=[E_n+F_n, H_i+H_{2n+1-i}]$ is only nonzero when $i=n$, in which case
\begin{align*}
[\mathbf{t},\mathbf{d}_{n-1}]&=[E_n+F_n,H_{n}+H_{n+1}]\\
&=[E_n,H_{n}]+[E_n,H_{n+1}]+[F_n,H_{n}]+[F_n,H_{n+1}]\\
&=E_n-E_n-F_n+F_n=0.
\end{align*}
Also, for $1\leq i \leq n$ and $1\leq j\leq n-1$, $2\leq i+j\leq 2n-1$ and
\begin{align*}
\mathbf{d}_i\mathbf{e}_j-\mathbf{e}_j\mathbf{d}_i=&[H_i+H_{2n+1-i},E_j+F_{2n-j}]\\
=&(\delta_{i,j}-\delta_{i,j+1}+\delta_{2n+1-i,j}-\delta_{2n+1-i,j+1})E_j\\
&+(-\delta_{i,2n-j}+\delta_{i,2n-j+1}-\delta_{2n+1-i,2n-j}+\delta_{2n+1-i,2n-j+1})F_{2n-j}\\
=&(\delta_{i,j}-\delta_{i,j+1})(E_j+F_{2n-j})=(\delta_{i,j}-\delta_{i,j+1})\mathbf{e}_j.
\end{align*}
Similarly,
\begin{align*}
\mathbf{d}_i\mathbf{f}_j-\mathbf{f}_j\mathbf{d}_i=&[H_i+H_{2n+1-i},F_j+E_{2n-j}]\\
=&(-\delta_{i,j}+\delta_{i,j+1}-\delta_{2n+1-i,j}+\delta_{2n+1-i,j+1})F_j\\
&+(\delta_{i,2n-j}-\delta_{i,2n-j+1}+\delta_{2n+1-i,2n-j}-\delta_{2n+1-i,2n-j+1})E_{2n-j}\\
=&(-\delta_{i,j}+\delta_{i,j+1})(F_j+E_{2n-j})=(-\delta_{i,j}+\delta_{i,j+1})\mathbf{f}_j.
\end{align*}
R2):
The only nontrivial relation to check is $\mathbf{e}_i\mathbf{f}_j-\mathbf{f}_j\mathbf{e}_i=\delta_{i,j}\mathbf{d}_i-\mathbf{d}_{i+1}$, and the calculation is as follows:
\begin{align*}
\mathbf{e}_i\mathbf{f}_j-\mathbf{f}_j\mathbf{e}_i=&[E_i+F_{2n-i},F_i+E_{2n-i}]\\
=&[E_i,F_i]-[E_{2n-i},F_{2n-i}]\\
=&H_{i}-H_{i+1}-H_{2n-i}+H_{2n-i+1}=\mathbf{d}_i-\mathbf{d}_{i+1}.
\end{align*}

R3):
\begin{align*}
&\mathbf{e}_{n-1}^2\mathbf{t}-2\mathbf{e}_{n-1}\mathbf{t}\mathbf{e}_{n-1}+\mathbf{t}\mathbf{e}_{n-1}^2 \\
&=(E_{n-1}+F_{n+1})^2(E_n+F_n)-2(E_{n-1}+F_{n+1})(E_n+F_n)(E_{n-1}+F_{n+1})\\
&-(E_n+F_n)(E_{n-1}+F_{n+1})^2,
\end{align*}
where
\begin{align*}
&(E_{n-1}+F_{n+1})^2E_n-2(E_{n-1}+F_{n+1})E_n(E_{n-1}+F_{n+1})+E_n(E_{n-1}+F_{n+1})^2\\
&=(E_{n-1}^2+2E_{n-1}F_{n+1}+F_{n+1}^2)E_n\\
&-2(E_{n-1}E_n E_{n-1}+F_{n+1}E_nE_{n-1}+E_{n-1}E_nF_{n+1}+F_{n+1}E_nF_{n+1})\\
&+E_n(E_{n-1}^2+2E_{n-1}F_{n+1}+F_{n+1}^2)\\
&=(E_{n-1}^2E_n-2E_{n-1}E_n E_{n-1}+E_nE_{n-1}^2)\\
&+(2E_{n-1}F_{n+1}E_n-2F_{n+1}E_nE_{n-1}-2E_{n-1}E_nF_{n+1}+2E_nE_{n-1}F_{n+1})\\
&+(F_{n+1}^2E_n-2F_{n+1}E_nF_{n+1}+E_nF_{n+1}^2)=0.
\end{align*}
Similarly,
\begin{align*}
(E_{n-1}+F_{n+1})^2F_n-2(E_{n-1}+F_{n+1})F_n(E_{n-1}+F_{n+1})+F_n(E_{n-1}+F_{n+1})^2=0.
\end{align*}
and therefore the relation holds. The other relation can be checked in a similar way.

R4): 
\begin{align*}
&\mathbf{e}_{n-1}\mathbf{t}^2-2\mathbf{t}\mathbf{e}_{n-1}\mathbf{t}+\mathbf{t}^2\mathbf{e}_{n-1}\\
&=(E_{n-1}+F_{n+1})(E_n+F_n)^2-2(E_n+F_n)(E_{n-1}+F_{n+1})(E_n+F_n)\\
&+(E_n+F_n)^2(E_{n-1}+F_{n+1}),
\end{align*}
where 
\begin{align*}
&E_{n-1}(E_n+F_n)^2-2(E_n+F_n)E_{n-1}(E_n+F_n)+(E_n+F_n)^2E_{n-1}\\
&=E_{n-1}(E_n^2+E_nF_n+F_nE_n+F_n^2)\\
&-2(E_nE_{n-1}E_n+E_nE_{n-1}F_n+F_nE_{n-1}E_n+F_nE_{n-1}F_n)\\
&+(E_n^2+E_nF_n+F_nE_n+F_n^2)E_{n-1}\\
&=(E_{n-1}E_n^2-2E_nE_{n-1}E_n+E_n^2E_{n-1})\\
&+(E_{n-1}F_n^2-2F_nE_{n-1}F_n+F_n^2E_{n-1})\\
&+E_{n-1}(E_nF_n+F_nE_n)-2(E_nE_{n-1}F_n+F_nE_{n-1}F_n)+(E_nF_n+F_nE_n)E_{n-1}\\
&=E_{n-1}(2F_nE_n+H_n-H_{n+1})-2(E_nF_nE_{n-1}+E_{n-1}F_nF_n)\\
&+(2E_nF_n-(H_n-H_{n+1}))E_{n-1}\\
&=[E_{n-1},H_n]=E_{n-1}.
\end{align*}
On the other hand,
\begin{align*}
&F_{n+1}(E_n+F_n)^2-2(E_n+F_n)F_{n+1}(E_n+F_n)+(E_n+F_n)^2F_{n+1}\\
&=F_{n+1}(E_n^2+E_nF_n+F_nE_n+F_n^2)\\
&-2(E_nF_{n+1}E_n+E_nF_{n+1}F_n+F_nF_{n+1}E_n+F_nF_{n+1}F_n)\\
&+(E_n^2+E_nF_n+F_nE_n+F_n^2)F_{n+1}\\
&=(F_{n+1}F_n^2-2F_nF_{n+1}F_n+F_n^2F_{n+1})\\
&+(F_{n+1}E_n^2-2E_nF_{n+1}E_n+E_n^2F_{n+1})\\
&+F_{n+1}(E_nF_n+F_nE_n)-2(E_nF_{n+1}F_n+F_nF_{n+1}E_n)+(E_nF_n+F_nE_n)F_{n+1}\\
&=F_{n+1}(2E_nF_n-(H_n-H_{n+1}))-2(F_{n+1}E_nF_n+F_nE_nF_{n+1})\\
&+(2F_nE_n+H_n-H_{n+1})F_{n+1}\\
&=[F_{n+1},H_{n+1}]=F_{n+1}.
\end{align*}
Therefore ,
\begin{align*}
\mathbf{e}_{n-1}\mathbf{t}^2-2\mathbf{t}\mathbf{e}_{n-1}\mathbf{t}+\mathbf{t}^2\mathbf{e}_{n-1}=E_{n-1}+F_{n+1}=\mathbf{e}_{n-1}.
\end{align*}
The other relation can be checked similarly.

The fact that these are the only relations follows from the proof of Proposition~\ref{presentation}.
\end{proof}

\begin{rem}
The defining relations in Proposition~\ref{presentation2} can be deduced from those in Proposition~\ref{presentation} by exploring the compatibility of a certain imbedding 
$\mathfrak{gl}_{N} \to \mathfrak{gl}_{N+1}$ and the involutions on both sides. 
\end{rem}

\subsection{An algebra homomorphism between $U^i$ and $U(\mathfrak{gl}_n\oplus \mathfrak{gl}_n)$}\label{S3.2}

Denote $U^i=U(\mathfrak{g}^{\theta})$ based on the parity assumption in this section. On the other hand, let $e_1,\dots,e_{n-1}$, $f_1,\dots,f_{n-1}$, $h_1,\dots,h_n$ be the Chevalley generators of the first copy of $\mathfrak{gl}_n(\mathbb{C})$ in $\mathfrak{gl}_n(\mathbb{C})\oplus \mathfrak{gl}_n(\mathbb{C})$, and $e_{\overline{1}},\dots,e_{\overline{n-1}}$, $f_{\overline{1}},\dots,f_{\overline{n-1}}$, $h_{\overline{1}},\dots,h_{\overline{n}}$ the Chevalley generators of the second copy.
\begin{lemma}
There is an algebra homomorphism: $\rho: U^i \to U(\mathfrak{gl}_n \oplus \mathfrak{gl}_n)$ via the following assignment
\begin{align*}
\mathbf{e}_i &\mapsto e_i+e_{\overline{i}} \hspace{.2 in}(1\leq i\leq n-1),\\
\mathbf{f}_i &\mapsto f_i+f_{\overline{i}}\hspace{.2 in} (1\leq i\leq n-1),\\
\mathbf{d}_i &\mapsto h_i+h_{\overline{i}} \hspace{.2 in} (1\leq i\leq n),\\
\mathbf{t} & \mapsto h_n-h_{\overline{n}}.
\end{align*}
\end{lemma}
\begin{proof}
We check that all the relations are satisfied in $U(\mathfrak{gl}_n \oplus \mathfrak{gl}_n)$ after applying $\rho$.

(R1) The first relation becomes $[h_i+h_{\overline{i}},h_n+h_{\overline{n}}]=0$ which is true. The other two relations are straightforward to check.

(R2) becomes $[e_i+e_{\overline{i}}, h_n-h_{\overline{n}}]=0$ for $1\leq i\leq n-2$ which is true. The other relation in (R2) are straightforward to check.

For (R3), we have
\begin{align*}
&\rho(\mathbf{e}_{n-1}^2\mathbf{t}-2\mathbf{e}_{n-1}\mathbf{t}\mathbf{e}_{n-1}+\mathbf{t}\mathbf{e}_{n-1}^2)\\
&=(e_{n-1}+e_{\overline{n-1}})^2(h_{n}-h_{\overline{n}})-2(e_{n-1}+e_{\overline{n-1}})(h_{n}-h_{\overline{n}})(e_{n-1}+e_{\overline{n-1}})+(h_{n}-h_{\overline{n}})(e_{n-1}+e_{\overline{n-1}})^2.
\end{align*}
Since generators of the first block commute with those of the second block, it yields
\begin{align*}
&(e_{n-1}+e_{\overline{n-1}})^2h_{n}-2(e_{n-1}+e_{\overline{n-1}})h_{n}(e_{n-1}+e_{\overline{n-1}})+h_{n}(e_{n-1}+e_{\overline{n-1}})^2\\
&=(e_{n-1}^2+2e_{n-1}e_{\overline{n-1}}+e_{\overline{n-1}}^2)h_{n}\\
&-2(e_{n-1}h_ne_{n-1}+e_{\overline{n-1}}h_{n}e_{n-1}+e_{n-1}h_ne_{\overline{n-1}}+e_{\overline{n-1}}h_ne_{\overline{n-1}})\\
&+h_n(e_{n-1}^2+2e_{n-1}e_{\overline{n-1}}+e_{\overline{n-1}}^2)\\
&=e_{n-1}^2h_n-2e_{n-1}h_ne_{n-1}+h_ne_{n-1}^2\\
&=e_{n-1}(h_ne_{n-1}-[h_n,e_{n-1}])-2e_{n-1}h_ne_{n-1}+(e_{n-1}h_n+[h_n,e_{n-1}])e_{n-1}=0.
\end{align*}
By symmetry, 
\begin{align*}
(e_{n-1}+e_{\overline{n-1}})^2h_{\overline{n}}-2(e_{n-1}+e_{\overline{n-1}})h_{\overline{n}}(e_{n-1}+e_{\overline{n-1}})+h_{\overline{n}}(e_{n-1}+e_{\overline{n-1}})^2=0.
\end{align*}
and so (R3) is satisfied. The other relation in (R3) can be checked similarly.

For (R4), we have
\begin{align*}
&\rho(\mathbf{e}_{n-1}\mathbf{t}^2-2\mathbf{t}\mathbf{e}_{n-1}\mathbf{t}+\mathbf{t}^2\mathbf{e}_{n-1})\\
&=(e_{n-1}+e_{\overline{n-1}})(h_{n}-h_{\overline{n}})^2-2(h_{n}-h_{\overline{n}})(e_{n-1}+e_{\overline{n-1}})(h_{n}-h_{\overline{n}})+(h_{n}-h_{\overline{n}})^2(e_{n-1}+e_{\overline{n-1}}).
\end{align*}
Since generators of different blocks commute, it becomes
\begin{align*}
&e_{n-1}(h_{n}-h_{\overline{n}})^2-2(h_{n}-h_{\overline{n}})e_{n-1}(h_{n}-h_{\overline{n}})+(h_{n}-h_{\overline{n}})^2e_{n-1}\\
&=e_{n-1}(h_{n}^2-2h_{n}h_{\overline{n}}+h_{\overline{n}}^2)\\
&-2(h_{n}e_{n-1}h_{n}-h_{\overline{n}}e_{n-1}h_n-h_ne_{n-1}h_{\overline{n}}+h_{\overline{n}}e_{n-1}h_{\overline{n}})\\
&+(h_{n}^2-2h_{n}h_{\overline{n}}+h_{\overline{n}}^2)e_{n-1}\\
&=e_{n-1}h_{n}^2-2h_{n}e_{n-1}h_{n}+h_{n}^2e_{n-1}\\
&=(h_ne_{n-1}+e_{n-1})h_n-2h_{n}e_{n-1}h_{n}+h_n(e_{n-1}h_n-e_{n-1})\\
&=[e_{n-1},h_n]=e_{n-1}.
\end{align*}
On the other hand, there is
\begin{align*}
&e_{\overline{n-1}}(h_{n}-h_{\overline{n}})^2-2(h_{n}-h_{\overline{n}})e_{\overline{n-1}}(h_{n}-h_{\overline{n}})+(h_{n}-h_{\overline{n}})^2e_{\overline{n}-1}\\
&=e_{\overline{n-1}}(h_{n}^2-2h_{n}h_{\overline{n}}+h_{\overline{n}}^2)\\
&-2(h_{n}e_{\overline{n-1}}h_{n}-h_{\overline{n}}e_{\overline{n-1}}h_n-h_ne_{\overline{n-1}}h_{\overline{n}}+h_{\overline{n}}e_{\overline{n-1}}h_{\overline{n}})\\
&+(h_{n}^2-2h_{n}h_{\overline{n}}+h_{\overline{n}}^2)e_{\overline{n-1}}\\
&=e_{\overline{n-1}}h_{\overline{n}}^2-2h_{\overline{n}}e_{\overline{n-1}}h_{\overline{n}}+h_{\overline{n}}^2e_{\overline{n-1}}\\
&=(h_{\overline{n}}e_{\overline{n-1}}+e_{\overline{n-1}})h_{\overline{n}}-2h_{\overline{n}}e_{\overline{n-1}}h_{\overline{n}}+h_{\overline{n}}(e_{\overline{n-1}}h_{\overline{n}}-e_{\overline{n-1}})\\
&=[e_{\overline{n-1}},h_{\overline{n}}]=e_{\overline{n-1}}.
\end{align*}
Therefore, we must have
\begin{align*}
&\rho(\mathbf{e}_{n-1}\mathbf{t}^2-2\mathbf{t}\mathbf{e}_{n-1}\mathbf{t}+\mathbf{t}^2\mathbf{e}_{n-1})=e_{n -1}+e_{\overline{n-1}}=\rho(\mathbf{e}_{n-1}).
\end{align*}
The other relation in R4) can be checked similarly. Lemma is proved.
\end{proof}

\subsection{Compatibility of two actions}
In this section we will see that the homomorphism $\rho$ defined in Section~\ref{S3.2} is compatible with the actions of $U^i$ and $U(\mathfrak{gl}_n\oplus \mathfrak{gl}_n)$ on $V^{\otimes}$. In particular, there is a vector space automorphism on $V^{\otimes d}$ which intertwines these two actions. We will also show that this linear map intertwines two (ostensibly different) Weyl group actions, which are known to be Schur-Weyl dual of $U^i$ and $U(\mathfrak{gl}_n\oplus \mathfrak{gl}_n)$, respectively. Let $V$ be the vector space over $\mathbb{C}$ with orderd basis $\mathcal{B}_1=\{v_1,\dots,v_n,v_{\overline{n}},\dots,v_{\overline{1}}\}$. Since $V$ is $2n$-dimensional, it is isomorphic to the $\mathbb{C}$-vector space with ordered basis $\mathcal{B}_2=\{w_1,\dots,w_n,w_{\overline{1}},\dots,w_{\overline{n}}\}$. Let $L:V\to V$ be the linear map such that $L(v_i)=w_i+w_{\overline{i}}$, $L(v_{\overline{i}})=w_i-w_{\overline{i}}$ for $1\leq i\leq n$. Notice $L$ is an isomorphism if and only if the characteristic of the underlying field is odd. Let $L^{\otimes d}:V^{\otimes d}\to V^{\otimes d}$ be the linear map such that
\begin{align*}
L^{\otimes d}(u_1\otimes \cdots\otimes u_d)=L(u_1)\otimes \cdots \otimes L(u_d) \hspace{.3 in} u_i\in V, \forall i
\end{align*}
We first show that $L^{\otimes d}$ intertwines two Weyl group actions. The type B Weyl group $B_d$ is the group with generators $s_1,\dots, s_{d}$, subject to relations
\begin{align*}
&s_i^2=1 \hspace{.2 in} (1\leq i,j\leq d) \hspace{.5 in}
s_is_j=s_js_i  \hspace{.2 in}  (|i-j|>1),\\
&s_is_js_i=s_js_is_j  \hspace{.2 in}(|i-j|=1, 1\leq i,j\leq d-1),\\
&s_{d-1}s_ds_{d-1}s_d=s_ds_{d-1}s_ds_{d-1}.
\end{align*}
The first action of $B_d$ on $V^{\otimes d}$, denoted by $f_1:\mathbb{C}B_d\to \operatorname{End}(V^{\otimes d})$, was first given by Green \cite{Green97} and defined as follows:
\begin{align*}
f_1(s_i) ( v_{k_1} \otimes  \cdots \otimes v_{k_i} \otimes v_{k_{i+1}} \otimes \cdots \otimes v_{k_d}) =&   v_{k_1} \otimes  \cdots \otimes v_{k_{i+1}} \otimes v_{k_{i}} \otimes \cdots \otimes v_{k_d} \\
&(1\leq i\leq d-1),\\
f_1(s_d) ( v_{k_1} \otimes \cdots \otimes v_{k_d}) =& v_{k_1} \otimes \cdots \otimes v_{\overline{k_d}}.
\end{align*} 
Here, $k_1,\dots, k_d \in \{1,\dots,n,\overline{n},\dots,\overline{1}\}$, and $\overline{\overline{i}}=i$. 

On the other hand, another action $f_2$, was defined in Hu-Stoll \cite{HS04} and Mazorchuk-Stroppel \cite{MS16}, with its quantum version in Shoji-Sakamoto \cite{SS99}. Using the basis element $w_i$, this is the action below:
\begin{align*}
f_2(s_i) ( w_{k_1} \otimes  \cdots \otimes w_{k_i} \otimes w_{k_{i+1}} \otimes \cdots \otimes w_{k_d}) =&   w_{k_1} \otimes  \cdots \otimes w_{k_{i+1}} \otimes w_{k_{i}} \otimes \cdots \otimes w_{k_d} \\
&(1\leq i\leq d-1),\\
f_2(s_d) ( w_{k_1} \otimes \cdots \otimes w_{k_d}) =& (-1)^{\epsilon(k_d)}
w_{k_1} \otimes \cdots \otimes w_{{k_d}},
\end{align*} 
where $k_1,\dots, k_d \in \{1,\dots,n,\overline{1},\dots,\overline{n}\}$, and  $\epsilon(k_d)$ is $1$ if $k_d$ is unbarred, and $-1$ if $k_d$ is barred.

\begin{prop}\label{intertwinesweyl}
The automorphism $L^{\otimes d}$ on $V^{\otimes d}$ induced by $L$, intertwines the two actions of the Weyl group $B_d$, in the following sense
\begin{align}
L^{\otimes d} (f_1(s)(u))=f_2(s)(L^{\otimes d}(u)) \hspace{.5 in} \forall s\in B_d, u\in V^{\otimes d}. \label{compatible}
\end{align}
In other words, the following diagram holds
\[
\begin{CD}
 @.  V^{\otimes d} @. \curvearrowleft^{f_1}  @. S^B_d\\
  @.  @VV  L^{\otimes d} V @. @VV= V \\
 @.  V^{\otimes d} @. \curvearrowleft^{f_2}  @. S^B_d\\
\end{CD}
\]
\end{prop}

\begin{proof}
It is enough to show equation (\ref{compatible}) holds for any $s_i$ with  $1\leq i\leq d$ and when $u$ is any pure tensor. When $1 \leq i\leq d-1$, this is a straightforward check. When $i=d$:

1) When $u=v_{i_1}\otimes \cdots \otimes v_{i_d}$, where $i_1,\dots,i_{d-1}$ can be either barred or unbarred, and $i_d$ is unbarred,
\begin{align*}
L^{\otimes d} (f_1(s_d)(u))=&L^{\otimes d}(v_{i_1}\otimes \cdots\otimes v_{i_{d-1}}\otimes v_{\overline{i_d}})\\
=&L(e_{i_1})\otimes \cdots \otimes L(e_{i_d})\otimes (w_{i_d}-w_{\overline{i_d}})\\
f_2(s_d)(L^{\otimes d}(u))=& f_2(s_d)(L(v_{i_1})\otimes \cdots\otimes L(v_{i_{d-1}})\otimes (w_{i_d}+w_{\overline{i_d}}))\\
=&L(e_{i_1})\otimes \cdots \otimes L(e_{i_d})\otimes (w_{i_d}-w_{\overline{i_d}})
\end{align*}
2) On the other hand, when $u=v_{i_1}\otimes \cdots \otimes v_{\overline{i_d}}$, where $i_1,\dots,i_{d-1}$ can be either barred or unbarred, and $i_d$ is unbarred,
\begin{align*}
L^{\otimes d} (f_1(s_d)(u))=&L^{\otimes d}(v_{i_1}\otimes \cdots\otimes v_{i_{d-1}}\otimes v_{{i_d}})\\
=& L(e_1)\otimes \cdots L(e_{i_{d-1}})\otimes (w_{i_d}+w_{\overline{i_d}})\\
f_2(s_d)(L^{\otimes d}(u))=&f_2(s_d)(L(e_{i_1})\otimes \cdots \otimes L(e_{i_{d-1}})\otimes (w_{i_d}-w_{\overline{i_d}}))\\
=&L(e_{i_1})\otimes \cdots\otimes L(e_{i_{d-1}})\otimes (w_{i_d}+w_{\overline{i_d}})
\end{align*}
\end{proof}

The space $V$ can be regarded as the defining representation for $\mathfrak{gl}(V)=\mathfrak{gl}_{2n}(\mathbb{C})$, using the ordered basis $\mathcal{B}_1$. For future convenience we will relabel the barred vectors as 
\begin{align*}
\mathcal{B}_1=\{v_1,\dots,v_n,v_{n+1}=v_{\overline{n}},\dots,v_{2n}=v_{\overline{1}}\}.
\end{align*}
 As a subalgebra of $\mathfrak{gl}_{2n}(\mathbb{C})$, $\mathfrak{gl}_{2n}^{\theta}$ also acts on $V$, and denote this action as $g_1: U(\mathfrak{gl}_{2n}^{\theta})\to \operatorname{End}(V)$. It induces an action on  $V^{\otimes d}$, denoted as $g_1$ by an abuse of notation, via 
\begin{align*}
g_1(x)(v_{i_1}\otimes \cdots \otimes v_{i_d}) = g_1(x)(v_{i_1})\otimes \cdots \otimes g_1(x)(v_{i_d}).
\end{align*}
This action is known to commute with the action $f_1$ of $B_d$ on $V^{\otimes d}$, when $\mathfrak{gl}_{2n}^{\theta}$ gets replaced by its group analogue in \cite{Green97}, or in the quantum case in \cite{BKLW18}.

Using the ordered basis $\mathcal{B}_2$, one can define an action of $\mathfrak{gl}_n(\mathbb{C})\oplus \mathfrak{gl}_n(\mathbb{C})$ on $V$, by regarding the vector space $V_1$ spanned by $\{w_1,\dots,w_n\}$ as the defining representation for the first copy $\mathfrak{gl}(V_1)=\mathfrak{gl}_n$ and the vector space $V_{-1}$ spanned by $\{w_{\overline{1}},\dots,w_{\overline{n}}\}$ as the defining representation for the second copy $\mathfrak{gl}(V_{-1})=\mathfrak{gl}_n$. We denote this action as $g_2 : U(\mathfrak{gl}_n(\mathbb{C})\oplus \mathfrak{gl}_n(\mathbb{C}))\to \operatorname{End}(V)$, and by an abuse of notation also denote by $g_2$ the induced action on $V^{\otimes d}$. This action is known to commute with the $f_2$ action of $B_d$ on $V^{\otimes d}$ (see \cite{HS04,MS16}.)

\begin{prop}\label{intertwinesenveloping}
The automorphism $L^{\otimes d}$ intertwines the action of $U^i$ and $U(\mathfrak{gl}_n\oplus\mathfrak{gl}_n)$, with respect to the homomorphism $\rho$, in the following sense
\begin{align*}
g_2(\rho(x))(L^{\otimes d}(u))=L^{\otimes d}(g_1(x)(u)) \hspace{.5 in} \forall x\in U^i,u\in V^{\otimes d}.
\end{align*}
In other words, the following diagram commutes

\[
\begin{CD}
U^i  @. \curvearrowright^{g_1}  @.  V^{\otimes d} \\
@V\rho  VV  @.  @VV L^{\otimes d} V  \\
U(\mathfrak{gl}_n\oplus\mathfrak{gl}_n) @. \curvearrowright^{g_2}  @.  V^{\otimes d} \\
\end{CD}
\]

\end{prop}

\begin{proof}
It is enough to check the following holds for any $x$ that is a generator among $\mathbf{e}_i, \mathbf{f}_i, \mathbf{t}, \mathbf{d}_i$, and for any $v \in V$.
\begin{align}
g_2(\rho(x))(L(v))=L(g_1(x)(v)). \label{intertwineonV}
\end{align}
The main claim then follows from extending the action of $U^i$ and $U(\mathfrak{gl}_n\oplus\mathfrak{gl}_n)$ from $V$ to $V^{\otimes d}$:

1) To check that (\ref{intertwineonV}) holds for $x=\mathbf{e}_j$, $1\leq j\leq n-1$, we first check it when $v=v_i$ for $1\leq i\leq n$:
\begin{align*}
&g_2(\rho(\mathbf{e}_j))L(v_i)=g_2(\rho(\mathbf{e}_j))(w_i+w_{\overline{i}})=g_2(e_j+e_{\overline{j}})(w_i+w_{\overline{i}})=\delta_{j+1,i}(w_{i-1}+w_{\overline{i-1}}),\\
&L(g_1(\mathbf{e}_j)(v_i))=L(g_1(E_j+F_{2n-j})v_i)=L(\delta_{j+1,i}v_{i-1})=\delta_{j+1,i}(w_{i-1}+w_{\overline{i-1}}).
\end{align*}
Here, the last equality holds because $n+1\leq 2n-j\leq 2n-1$ and $F_{2n-j}=E_{2n-j+1,2n-j}$,  where $E_{pq}$ is the matrix unit with $1$ in the $(p,q)$-position and $0$ elsewhere. Since $2n-j\neq i$, the action of $F_{2n-j}$ is zero on $v_i$. 

When $v=v_{\overline{i}}$ for $1\leq i\leq n$:
\begin{align*}
&g_2(\rho(\mathbf{e}_j))(L(v_{\overline{i}}))=g_2(\rho(\mathbf{e}_j))(w_i-w_{\overline{i}})=g_2(e_j+e_{\overline{j}})(w_i-w_{\overline{i}})=\delta_{j+1,i}(w_{i-1}-w_{\overline{i-1}}),\\
&L(g_1(\mathbf{e}_j)(v_{\overline{i}}))=L(g_1(E_j+F_{2n-j})(v_{2n+1-i}))=L(\delta_{2n-j,2n+1-i}v_{2n+2-i}),\\
&=\delta_{i,j+1}L(v_{\overline{i-1}})=\delta_{i,j+1}(w_{i-1}-w_{\overline{i-1}}).
\end{align*}
The second equality holds because $n+1\leq 2n+1-i\leq 2n$, $2\leq j+1\leq n$, $j+1\neq 2n+1-i$. Hence the action of $E_j=E_{j,j+1}$ is zero on $v_{2n+1-i}$.

2) We now check that (\ref{intertwineonV}) holds for $x=\mathbf{f}_j$, $1\leq j\leq n-1$ and $v=v_i$, $1\leq i\leq n$:
\begin{align*}
&g_2(\rho(\mathbf{f}_j))(L(v_i))=g_2(\rho(\mathbf{f}_j))(w_i+w_{\overline{i}})=g_2(f_j+f_{\overline{j}})(w_i+w_{\overline{i}})=\delta_{j,i}(w_{i+1}+w_{\overline{i+1}}),\\
&g_1(\mathbf{f}_j)(v_i)=g_1(F_j+E_{2n-j})(v_i)=\delta_{j,i}v_{i+1},\\
&L(g_1(\mathbf{f}_j)(v_i))=\delta_{j,i}(w_{i+1}+w_{\overline{i+1}}).
\end{align*}
Here, $E_{2n-j}$ acts on $v_i$ by zero. For a more careful discussion, please refer to case 1).

When $v=v_{\overline{i}}$, $1\leq i\leq n$,
\begin{align*}
&g_2(\rho(\mathbf{f}_j))L(v_{\overline{i}})=g_2(f_j+f_{\overline{j}})(w_i-w_{\overline{i}})=\delta_{j,i}(w_{i+1}-w_{\overline{i+1}}),\\
&g_1(\mathbf{f}_j)(v_{\overline{i}})=g_1(F_j+E_{2n-j})(v_{2n+1-i})=\delta_{2n+1-j,2n+1-i}v_{2n-i}=\delta_{i,j}v_{\overline{i+1}},\\
&L(g_1(\mathbf{f}_j)(v_{\overline{i}}))=\delta_{i,j}(w_{i+1}-w_{\overline{i+1}}).
\end{align*}
Here, $F_j$ acts on $v_{2n+1-i}$ by zero.

3) We now check the case when $x=\mathbf{t}$, $v=v_i$ for $1\leq i\leq n$:
\begin{align*}
&g_2(\rho(\mathbf{t}))L(v_i)=g_2(h_n-h_{\overline{n}})(w_i+w_{\overline{i}})=\delta_{n,i}(w_i-w_{\overline{i}}),\\
& g_1(\mathbf{t})=(E_n+F_n)v_i=\delta_{n+1,i}v_{i-1}+\delta_{n,i}v_{i+1}=\delta_{n+1,i}v_n+\delta_{n,i}v_{n+1}\\
&=\delta_{n+1,i}v_n+\delta_{n,i}v_{\overline{n}}=\delta_{n,i}v_{\overline{n}},\\
&L(g_1(\mathbf{t})(v_i))=\delta_{n,i}(w_n-w_{\overline{n}}).
\end{align*}
When $v=v_{\overline{i}}$ for $1\leq i\leq n$:
\begin{align*}
&g_2(\rho(\mathbf{t}))L(v_{\overline{i}})=g_2(h_n-h_{\overline{n}})(w_i-w_{\overline{i}})=\delta_{i,n}(w_n+w_{\overline{n}}),\\
&g_1(\mathbf{t})(v_{\overline{i}})=g_1(E_n+F_n)(v_{2n+1-i})=\delta_{n+1,2n+1-i}v_{2n-i}=\delta_{i,n}v_n,\\
& L(g_1(\mathbf{t})(v_{\overline{i}}))=\delta_{i,n}(w_n+w_{\overline{n}}).
\end{align*}

4) We now prove the case when $x=\mathbf{d}_j$ and $v=v_i$, for $1\leq i,j\leq n$:
\begin{align*}
&g_2(\rho(\mathbf{d}_j))(L(v_i))=g_2(h_j+h_{\overline{j}})(w_i+w_{\overline{i}})=\delta_{i,j}(w_i+w_{\overline{i}}),\\
&g_1(\mathbf{d}_j)(v_i)=g_1(H_j+H_{2n+1-j})(v_i)=\delta_{i,j}v_i,\\
&L(g_1(\mathbf{d}_j)(v_i))=\delta_{i,j}(w_i+w_{\overline{i}}).
\end{align*}
When $v=v_{\overline{i}}$, $1\leq i\leq n$:
\begin{align*}
&g_2(\rho(\mathbf{d}_j))(L(v_{\overline{i}}))=g_2(h_j+h_{\overline{j}})(w_i-w_{\overline{i}})=\delta_{i,j}(w_i-w_{\overline{i}}),\\
&g_1(\mathbf{d}_j)(v_{\overline{i}})=g_1(H_j+H_{2n+1-j})(v_{2n+1-i})=\delta_{i,j}v_{2n+1-i}=\delta_{i,j}v_{\overline{i}},\\
&L(g_1(\mathbf{d}_j)(v_{\overline{i}}))=\delta_{i,j}(w_i-w_{\overline{i}}).
\end{align*}
Hence we have checked all cases.
\end{proof}

\subsection{Root vectors in $U^i$}
We also introduce the notion of root vectors for $\mathfrak{gl}_{2n}(\mathbb{C})^{\theta}$ as we did prior to Lemma~\ref{independentofbrackets}. Similar to the notation earlier, we distinguish between the first and second copy in $\mathfrak{gl}_n(\mathbb{C})\oplus\mathfrak{gl}_n(\mathbb{C})$ by letting $\mathfrak{g}_1=\mathfrak{gl}_n(\mathbb{C})$, $\mathfrak{g}_2=\mathfrak{gl}_n(\mathbb{C})$, and for $\alpha\in \Phi_n$ a root in $\mathfrak{gl}_n$, let $Z_{\alpha}\in\mathfrak{g}_1,\overline{Z}_{\alpha}\in \mathfrak{g}_2$ be the associated root vectors. We adopt the previous notation that $\epsilon_i=\mathbf{d}_i^*$ for $1\leq i \leq n$ and $\mu_i=H_i^*$ for $1\leq i\leq 2n$. Let 
\begin{align*}
\Phi_{2n}^+=&\{\mu_i-\mu_j\hspace{.1 in}| \hspace{.1 in} 1\leq i<j\leq 2n \},\\
\Pi_{2n}=&\{\mu_i-\mu_{i+1}\hspace{.1 in}| \hspace{.1 in} 1\leq i\leq 2n-1 \},
\end{align*}
be the set of positive roots and simple roots in $\mathfrak{gl}_{2n}(\mathbb{C})$, respectively. For any $\alpha\in \Phi_{2n}^+$, define root vectors $X_{\alpha}$ recursively as follows. For the simple roots:
\begin{align*}
X_{\mu_i-\mu_{i+1}}=&\mathbf{e}_i \hspace{.2 in} (1\leq i\neq n-1),\\
X_{\mu_n-\mu_{n+1}}=&\mathbf{t} \hspace{.2 in} (1\leq i\neq n-1),\\
X_{\mu_i-\mu_{i+1}}=&\mathbf{f}_{2n-i} \hspace{.2 in} (n+1\leq i\neq 2n-1).
\end{align*}
On $\alpha$ that is a sum of simple roots, $X_{\alpha}$ is defined recursively as (\ref{rootvectorsinfixed}). Versions of Lemmas~\ref{independentofbrackets} through \ref{span} is also true for root vectors in $\mathfrak{g}^{\theta}$. In particular, we have the following.
\begin{lemma}
The root vectors $\{X_{\alpha}\}_{\alpha\in \Phi^+_{2n}}$, $\{\mathbf{d}_i\}_{1\leq i\leq n}$ and $\{\mathbf{t}\}$, together span $\mathfrak{g}^{\theta}$.
\end{lemma}

For $\alpha\in\Phi_{n}$, an element $x\in \mathfrak{g}^{\theta}$ is a weight vector of weight $\alpha$, under the adjoint action of $\mathbf{d}_1,\dots,\mathbf{d}_n$, if and only if 
\begin{align*}
[\mathbf{d}_i,x]=\alpha(\mathbf{d}_i)x \hspace{.2 in} (1\leq i\leq n).
\end{align*}

Define a map $t: \Phi_{2n}^+\to \Phi_{n}\cup\{0\}$, such that on the simple roots $\alpha \in \Pi_{2n}$, $X_{\alpha}$ is a weight vector of weight $t(\alpha)$. That is to say,
\begin{align*}
t(\mu_i-\mu_{i+1})=&\epsilon_i-\epsilon_{i+1} \hspace{.2 in} (1\leq i\leq n-1), \\
t(\mu_n-\mu_{n+1})=&0, \\ 
t(\mu_i-\mu_{i+1})=&-\epsilon_{2n-i}+\epsilon_{2n+1-i}  \hspace{.2 in} (n+1\leq i\leq 2n-1).
\end{align*} 
Extend $t$ linearly so that on a sum of simple roots $\alpha=\alpha_1+\cdots \alpha_i \in \Phi_{2n}^+$, define 
\begin{align*}
t(\alpha)=t(\alpha_1)+\cdots+t(\alpha_i).
\end{align*}

Then a version of Lemma~\ref{correctweight} is also true. Moreover, the stalks of $t$ are described below.
\begin{lemma}
For the map $t$ defined above,
\begin{align*}
t^{-1}(0)=&\{\mu_i-\mu_{2n+1-i} \hspace{.1 in} | \hspace{.1 in} 1\leq i \leq n \}, \\
t^{-1}(\epsilon_i-\epsilon_n)=&\{\mu_i-\mu_n\} \hspace{.5 in} (1\leq i \leq n-1),\\
t^{-1}(-\epsilon_i+\epsilon_n)=&\{\mu_n-\mu_{2n+1-i}\} \hspace{.5 in} (1\leq i \leq n-1),\\
t^{-1}(\epsilon_i-\epsilon_j)=&\{\mu_i-\mu_j,\mu_i-\mu_{2n+1-j}\}  \hspace{.5 in} (1\leq i<j \leq n-1),\\
t^{-1}(-\epsilon_i+\epsilon_j)=&\{\mu_j-\mu_{2n+1-i}, \mu_{2n+1-j}-\mu_{2n+1-i}\}  \hspace{.5 in} (1\leq i<j \leq n-1).
\end{align*}
\end{lemma}
\begin{proof}
The proof is very similar to that of Proposition~\ref{stalksofweights} and is left to the reader.
\end{proof}

\subsection{An algebra isomorphism with $U(\mathfrak{gl}_n\oplus\mathfrak{gl}_n)$}
We now re-index the root vectors using their weights. Specifically, define $\mathbf{t}_1,\dots,\mathbf{t}_{n-1},\mathbf{t}_n$ such that  $\mathbf{t}_n=\mathbf{t}$,
\begin{align}
\mathbf{t}_{i-1}=\mathbf{t}_i+X_{\mu_{i-1}-\mu_{2n+1-i}} \hspace{.2 in} (2\leq i\leq n). \label{morets}
\end{align}
For positive root vectors,
\begin{align*}
 X_{\epsilon_i-\epsilon_n}=&X_{\mu_i-\mu_n} \hspace{.2 in } (1\leq i\leq n-1),\\
 X_{\epsilon_i-\epsilon_j}=&X_{\mu_i-\mu_j}, \hspace{.2 in}  X'_{\epsilon_i-\epsilon_j}=X_{\mu_i-\mu_{2n+1-j}} \hspace{.2 in } (1\leq i<j\leq n-1).
\end{align*}

Recall that $e_i,f_i,h_i\in \mathfrak{g}_1$ and $e_{\overline{i}},f_{\overline{i}},h_{\overline{i}} \in\mathfrak{g}_2$ are Chevalley generators of each copy of $\mathfrak{gl}_n(\mathbb{C})$.
\begin{lemma}\label{rootvectorsfortheothercase}
Under the map $\rho$, for $\alpha \in \Phi_{n}$
\begin{align}
\rho(\mathbf{t}_i)=&h_i-h_{\overline{i}} \hspace{.2 in}(1\leq i\leq n), \label{anotherimage1}\\
\rho(X_{\alpha})=&Z_{\alpha}+\overline{Z}_{\alpha}, \label{anotherimage2} \\
\rho(X_{\alpha}')=&Z_{\alpha}-\overline{Z}_{\alpha}. \label{anotherimage3}
\end{align}
In particular, $\rho$ is an isomorphism between Lie algebras.
\end{lemma}
\begin{proof}
We first prove (\ref{anotherimage2}) for a positive root $\alpha=\epsilon_i-\epsilon_j$, by inducting on $i$. The base case when $i=j-1$ is given by the definition of $\rho$ on $\mathbf{e}_{j-1}=X_{\epsilon_{j-1}-\epsilon_{j}}$. Now we claim if (\ref{anotherimage1}) is true for $(i,j)$, it is also true for $(i-1,j)$:
\begin{align*}
\rho(X_{\epsilon_{i-1}-\epsilon_j})=&\rho([X_{\epsilon_{i-1}-\epsilon_i},X_{\ep_i-\ep_j}])=[\rho(\mathbf{e}_{i-1}),Z_{\ep_i-\ep_j}+\overline{Z}_{\ep_i-\ep_j}]\\
=&[Z_{\ep_{i-1}-\ep_i}+\overline{Z}_{\ep_{i-1}-\ep_{i}},Z_{\ep_i-\ep_j}+\overline{Z}_{\ep_i-\ep_j}]=Z_{\ep_{i-1}-\ep_j}+\overline{Z}_{\ep_{i-1}-\ep_{j}}.
\end{align*}
Now let us prove (\ref{anotherimage1}). The base case when $i=n$ follows from the definition of $\rho$ on $\mathbf{t}$. Suppose (\ref{anotherimage1}) is true for $i$, then we claim it is also true for $i-1$:
\begin{align*}
\rho(\mathbf{t}_{i-1})=&\rho([X_{\mu_{i-1}-\mu_{2n+2-i}})+\rho(\mathbf{t}_i)\\
=&\rho([X_{\mu_i-1-\mu_i},[X_{\mu_i-\mu_{2n+1-i}},X_{\mu_{2n+1-i}-\mu_{2n+2-i}}]])+\rho(\mathbf{t}_i)\\
=&\rho([\mathbf{e}_{i-1},[\mathbf{t}_i,X_{-\ep_{i-1}+\ep_i}]])+\rho(\mathbf{t}_i)=[\mathbf{e}_{i-1},[\mathbf{t}_{i},\mathbf{f}_{i-1}]]+\rho(\mathbf{t}_i)\\
=&[Z_{\ep_{i-1}-\ep_i}+\overline{Z}_{\ep_{i-1}-\ep_i},[h_i-h_{\overline{i}},Z_{-\ep_{i-1}+\ep_i}+\overline{Z}_{-\ep_{i-1}+\ep_i}]]+h_i-h_{\overline{i}}\\
=&[Z_{\ep_{i-1}-\ep_i}+\overline{Z}_{\ep_{i-1}-\ep_i},Z_{-\ep_{i-1}+\ep_i}-\overline{Z}_{-\ep_{i-1}+\ep_i}]+h_i-h_{\overline{i}}\\
=&h_{i-1}-h_i-h_{\overline{i-1}}+h_{\overline{i}}+h_i-h_{\overline{i}}=h_{i-1}-h_{\overline{i-1}}.
\end{align*}
Now we prove (\ref{anotherimage3}) by induction on $i$. The base case when $i=j-1$ (or $j=i+1$) is as follows
\begin{align*}
\rho(X_{\ep_i-\ep_{i+1}})=&\rho(X_{\mu_i-\mu_{2n-i}})=\rho([X_{\mu_i-\mu_{i+1}},X_{\mu_{i+1}-\mu_{2n-i}}])=\rho([\mathbf{e}_i,\mathbf{t}_{i+1}-\mathbf{t}_{i+2}])\\
=&[X_{\ep_i-\ep_{i+1}}+\overline{X}_{\ep_i-\ep_{i+1}},h_{i+1}-h_{\overline{i+1}}]=X_{\ep_i-\ep_{i+1}}-\overline{X}_{\ep_i-\ep_{i+1}}.
\end{align*}
Now suppose (\ref{anotherimage3}) is true for $(i,j)$ with $i<j<2n+1-j$, we claim it is also true for $(i-1,j)$:
\begin{align*}
\rho(X_{\ep_{i-1}-\ep_j})=&\rho(X_{\mu_{i-1}-\mu_{2n+1-j}})=\rho([X_{\mu_{i-1}-\mu_i},X_{\mu_i-\mu_{2n+1-j}}])=\rho([\mathbf{e}_i,X'_{\ep_i-\ep_j}])\\=&[X_{\ep_{i-1}-\ep_i}+\overline{X}_{\ep_{i-1}-\ep_i},X_{\ep_{i}-\ep_j}-\overline{X}_{\ep_{i}-\ep_j}]=X_{\ep_{i-1}-\ep_j}-\overline{X}_{\ep_{i-1}-\ep_j}.
\end{align*}
Hence we have proved the induction step for positive roots. The arguments for negative roots are similar.
\end{proof}

\subsection{The type B Schur algebra}
Similar to the previous case, fix a positive integer $d$, and let $S$ be the quotient of $U(\mathfrak{gl}_n(\mathbb{C})\oplus \mathfrak{gl}_n(\mathbb{C}))$ under the following relations
\begin{align*}
h_1+\cdots+h_{n}+h_{\overline{1}}+\cdots+h_{\overline{n}}&=d,\\
h_{i}(h_i-1)\cdots(h_{i}-d)&=0, \hspace{.5 in} 1\leq i\leq n,\\
h_{\overline{i}}(h_{\overline{i}}-1)\cdots(h_{\overline{i}}-d)&=0, \hspace{.5 in} 1\leq i\leq n .
\end{align*}
Further, recall that $\mathbf{t}_i$ is defined recursively using generators of $U^i$ in (\ref{morets}). Define the following element
\begin{align*}
\mathbf{h}_i=\frac{1}{2}(\mathbf{t}_i+\mathbf{d}_i), \hspace{.3 in}\mathbf{h}_{\overline{i}}=\frac{1}{2}(-\mathbf{t}_i+\mathbf{d}_i), \hspace{.5 in} (1\leq i\leq n).
\end{align*}

Then Lemma~\ref{rootvectorsfortheothercase} and the definition of $\rho$ implies that $\rho(\mathbf{h}_i)=h_i$ and $\rho(\mathbf{h}_{\overline{i}})=h_{\overline{i}}$. Therefore, let $S^i$  be the quotient of $U(\mathfrak{g}^{\theta})$ under the following relations

\begin{align}
\mathbf{d}_1+\cdots+\mathbf{d}_n=&d, \label{Sirelation1}\\
\h_i(\h_i-1)\cdots(\h_i-d)=&0, \hspace{.5 in} 1\leq i\leq n, \label{Sirelation2}\\
\h_{\overline{i}}(\h_{\overline{i}}-1)\cdots(\h_{\overline{i}}-d)=&0, \hspace{.5 in} 1\leq i\leq n. \label{Sirelation3}
\end{align}
Note: once unraveling the notation of root vectors, these relations are purely in terms of generators $\mathbf{e}_i$, $\mathbf{f}_i$, $\mathbf{t}$ and $\mathbf{d}_i$ of $U^i$.
\begin{prop}
The algebras $S$ and $S^i$ are isomorphic.
\end{prop}

\begin{example} When $n=1$, $S^i$ is the quotient of $\mathbb{C}[\mathbf{t}]$ under the following relations, where $d$ is a positive integer,
\begin{align*}
(\mathbf{t}+d)(\mathbf{t}+d-2)\cdots(\mathbf{t}-d+2)(\mathbf{t}-d)=&0.
\end{align*}
We refer to~\cite{Li19} for a presentation in the quantum case.
\end{example}

\begin{example} When $n=2$, $U(\mathfrak{gl}_{4}(\mathbb{C})^{\theta})$ is generated by $\mathbf{e},\f,\mathbf{t},\mathbf{d}_1,\mathbf{d}_2$.  Relations (\ref{Sirelation1}) - (\ref{Sirelation3}) become explicitly as follows.
\begin{align*}
\mathbf{d}_1+\mathbf{d}_2=&d,\\
([\e,\f]+\D_1)([\e,\f]+\D_1-2)\cdots([\e,\f]+\D_1-2d)=&0,\\
(-[\e,\f]+\D_1)(-[\e,\f]+\D_1-2)\cdots(-[\e,\f]+\D_1-2d)=&0,\\
(\mathbf{t}+\mathbf{d}_2)(\mathbf{t}+\mathbf{d}_2-2)\cdots(\mathbf{t}+\mathbf{d}_2-2d)=&0,\\
(-\mathbf{t}+\mathbf{d}_2)(-\mathbf{t}+\mathbf{d}_2-2)\cdots(-\mathbf{t}+\mathbf{d}_2-2d)=&0.
\end{align*}
\end{example}

\begin{rem}\label{generalfield}
Throughout this article, our choice of the ground field  $\mathbb{C}$ is based on the above desired result and the choice of ground field in the unpublished work of Kujawa-Zhu. Nevertheless, our results also hold for any algebraically closed field $\mathbf{k}$ with characteristic not equal to $2$, without alternation of the proofs.
\end{rem}


\end{document}